\documentclass[a4paper,11pt,reqno]{amsart}

\usepackage[latin1]{inputenc}\usepackage{amsmath,amssymb,amsthm,amsfonts,mathrsfs,amsopn}
\usepackage[all]{xy}
\usepackage{dsfont}
\usepackage{color}
\usepackage{esint}
\usepackage{mathtools}
\usepackage{bm}
\mathtoolsset{showonlyrefs}
\usepackage{slashed}
\usepackage{mathabx}
\usepackage{microtype}

\usepackage{tikz-cd}

\usepackage[colorlinks=true,urlcolor=blue, citecolor=red,linkcolor=blue,
linktocpage,pdfpagelabels, bookmarksnumbered,bookmarksopen]{hyperref}

\usepackage{geometry}
\usepackage[obeyDraft]{todonotes}

\newcommand{\Abracket}[1]{\left<#1\right>} 
\newcommand{\parenthesis}[1]{\left(#1\right)} 
\newcommand{\braces}[1]{\left\{#1\right\}} 

\newcommand{\R}{\mathbb{R}}
\newcommand{\C}{\mathbb{C}}

\newcommand{\cH}{\mathcal{H}}
\newcommand{\cO}{\mathcal{O}}

\newcommand{\cZ}{\mathcal{Z}}

\newcommand{\dd}{\mathop{}\!\mathrm{d}}
\newcommand{\eps}{\varepsilon}
\newcommand{\p}{\partial}
\newcommand{\sph}{\mathbb{S}}

\newcommand{\D}{\slashed{D}}
\newcommand{\G}{\slashed{G}}

\newcommand{\snabla}{\slashed{\nabla}} 

\numberwithin{equation}{section}

\newcommand{\CLip}{C_0}

\newcommand{\CCoord}{C_2}
\newcommand{\rCoord}{r_2}

\newcommand{\CHardy}{C_3}
\newcommand{\rHardy}{r_3}

\newcommand{\CPos}{C_4}
\newcommand{\rPos}{r_4}

\newcommand{\CAM}{C_{\textrm{am}}}
\newcommand{\rAM}{r_{\textrm{am}}}

\newcommand{\CbL}{C_{5}}

\newcommand{\Cq}{C_q}
\newcommand{\Cf}{C_f}

\DeclareMathOperator{\Cl}{Cl}

\DeclareMathOperator{\diag}{diag}

\DeclareMathOperator{\dv}{\dd{vol}}
\DeclareMathOperator{\ds}{\dd{s}}
\DeclareMathOperator{\dx}{dx}
\DeclareMathOperator{\diverg}{div}
\DeclareMathOperator{\dist}{dist}

\DeclareMathOperator{\End}{End}

\DeclareMathOperator{\id}{Id}

\DeclareMathOperator{\Span}{Span}
\DeclareMathOperator{\Scal}{Scal}

\DeclareMathOperator{\Spin}{Spin}

\DeclareMathOperator{\SO}{SO}

\newtheorem{thm}{Theorem}[section]
\newtheorem{dfn}[thm]{Definition}
\newtheorem{lemma}[thm]{Lemma}
\newtheorem{prop}[thm]{Proposition}
\newtheorem{cor}[thm]{Corollary}

\newtheorem{rmk}[]{Remark}

\begin{document}
\title{On the nodal set of solutions to Dirac equations}
\author[W. Borrelli and R. Wu]{William Borrelli$^{(1)}$ \& Ruijun Wu$^{(2)}$}
\addtocounter{footnote}{1}
\footnotetext{Dipartimento di Matematica, Politecnico di Milano, Piazza Leonardo da Vinci, 32, I-20133, Milano, Italy. E-mail address: {\tt william.borrelli@polimi.it}}
\addtocounter{footnote}{1}
\footnotetext{School of Mathematics and Statistics, Beijing Institute of Technology, Zhongguancun South St.5, 100081, Beijing, P.R.China. E-mail address:
{\tt ruijun.wu@bit.edu.cn}}

\begin{abstract} 
Motivated by various geometric problems, we study the nodal set of solutions to Dirac equations on manifolds, of general form. 
We prove that such set has Hausdorff dimension less than or equal to~$n-2$, $n$ being the ambient dimension. 
We extend this result, previously known only in the smooth case or in specific cases, working with locally Lipschitz coefficients.
Under some additional, but still quite general, structural assumptions we provide a stratification result for the nodal set, which appears to be new already in the smooth case. 
This is achieved by exploiting the properties of a suitable Almgren-type frequency function, which is of independent interest.
\end{abstract}
\maketitle

{\footnotesize
\emph{Keywords}: Dirac equation, nodal set, Hausdorff dimension, stratification, frequency function.

\medskip

\emph{2020 MSC}: 58J05, 35J46, 58J60.
}

\section{Introduction}\label{sec:intro}

We are concerned with the nodal set of spinors solving Dirac equations of the form
\begin{align}\label{eq:Dirac eq-V}
 \D_g\psi= V(\psi)
\end{align}
on a spin Riemannian manifold~$(M,g)$. 
The right hand side above takes different forms in various contexts, ranging from supersymmetric quantum field theory, conformal geometry, surface theory and submanifolds theory to positive mass theorems, spinorial Yamabe problems, spectral analysis of Dirac operators, and so on. Some examples are briefly sketched below. 
It is thus desirable to have good estimates about the nodal sets of such spinor solutions, especially in geometric applications, since nodal sets typically contain important geometric informations about the underlying manifold and on the solutions themselves. 
The analysis of the nodal sets in concrete backgrounds is carried out, for instance, in \cite{Ammann2003Habil, BMW2021ground, Branding2018nodal, Branding2018note, Xu2018conformal} and the only general results known to us was due to Christian B\"ar \cite{Baer1997nodal, Baer1999zero} which deals with \emph{smooth} right hand side, also considering more general first-order elliptic operators. 
He proves sharp  estimates for the Hausdorff dimension of the nodal set, and provides an upper bound for the Hausdorff densities.
However, in various geometric settings, especially in high dimensions, the right hand side is usually nonsmooth, so that B\"ar's results do not apply. This was the case, for instance, in \cite{BMW2021ground}, where classification of ground state solutions of the critical Dirac equation on the round sphere (also known as \emph{Dirac bubbles}) is considered. 
In that case a capacity argument is used in order to handle the nodal set. 
Generalizing the result of that paper was our original motivation for the present analysis.
 
Thus here we aim at studying the nodal set of spinors to~\eqref{eq:Dirac eq-V} with $V$ not necessarily smooth, extending to that case the dimension estimate available in the smooth case, as stated in Theorem \ref{thm:DimEstimate}. 
Moreover, under additional assumptions on $V$ we can also prove a stratification result in Theorem \ref{thm:Stratification}, which appears to be new also in the case of smooth potentials. 
This is proved by exploiting the properties of a suitable Almgren-type frequency function \eqref{eq:freq}, studied in the paper, which we think can be of independent interest. 
Its main property is stated in Theorem \ref{thm:AlmostMonotonicity}. Before stating our main results, we quickly recall some well-known examples to which they apply.

\subsection{Dirac equations in physics and geometry}

\subsubsection*{Spinorial Yamabe equation} 

This equation arises in the study of the spinorial Yamabe problem~\cite{ AGHM2008spinorial, AHM2006mass} and takes the form 
\begin{equation}\label{eq:spinYamabe}
 \D\psi = |\psi|^{\frac{2}{n-1}}\psi\,. 
\end{equation}
Its solutions are related to a conformal invariant analogous to the classical Yamabe invariant. 
Typically, one needs to make a conformal transformation with conformal factor proportional to $|\psi|^{\frac{4}{n-1}}$, see~\cite{Ammann2003Habil, BMW2021ground}, so that the zeros of the spinor~$\psi$ lead to a degeneracy of the metric, which is (ideally) to be excluded. 
Such singularity is removable, or negligible in the Sobolev sense, if the nodal set is sufficiently small, for instance, in capacity sense \cite{BMW2021ground}. For this reason, it is desirable to have fine estimates on the nodal set of the spinors. 
Note that the right hand side has the critical Sobolev nonlinearity. 
Moreover, for~$n\ge 2$, the nonlinearity on the right hand side \eqref{eq:spinYamabe} is non-smooth but superlinear, and for~$n\ge 4$ it is also sub-quadratic. 
This was our original motivation for the problem treated in this paper, and was addressed in \cite{BMW2021ground}.

We mention that generalizations of the spinorial  Yamabe equations are also widely studied in calculus of variations, e.g.~\cite{BartschXu2021spinorial,B2022sym, BF2020sharp, DingLiXu2016bifurcation, Grosse2012, SX2021inv}.  

\subsubsection*{Spinorial Weierstrass representation}

Another geometric application of Dirac equations is the famous Werierstrass representation of surfaces in~$R^3$ with prescribed mean curvature~$H\in C^\infty(M)$. 
It is well-known (see e.g.~\cite{Friedrich1998spinorrep, Matsutani1997cmc, Taimanov1997modified}) that such an immersion exists if and only if the Dirac equation 
\begin{align}\label{eq:H-Dirac}
 \D\psi = H\psi 
\end{align}
has a solution with constant length.
By a conformal transformation, we can put the equation into a nonlinear (indeed, superlinear) form similar to \eqref{eq:spinYamabe}, but then we only need the spinor to be nowhere vanishing, not necessarily of constant length. 
The spinorial Weierstrass representation has been greatly generalized to higher dimensional and codimensional cases, see~\cite{BBLR2023spinorial, BLR2017spinorial} and the references therein. 
The immersion induced from certain spinors satisfying various Dirac type equations still requires that the spinor to be non-vanishing. 
Even if we allow for branchings in the immersion, it is desirable that the nodal set is relatively small, say, of codimension at least two.

\subsubsection*{Eigenspinors and Yau's conjecture}

The eigenvalues of Dirac operators are of great importance in differential geometry and topology. 
Moreover, nodal set estimates for eigenspinors are also of interest in spectral geometry. 
A. Hermann \cite{Hermann2014zero} showed that on the sphere for a generic metric the eigenspinors have no zeros. 
Note that the eigenspinors can be dealt with using C. B\"ar's results~\cite{Baer1999zero}, since the right hand side is already smooth, and  the nodal set is rectifiable of codimension at least two. Our contribution, in this case, concerns the stratification of the nodal set.

The eigenfunctions of the Laplace operator on a Riemannian manifold have been studied for a long time, and there are also many recent results.  Among them, we mention the partial proof by A. Lugonov~\cite{Logunov2018nodal-upper, Logunov2018nodal-lower} of Yau's conjecture on the nodal sets, concerning lower and upper bounds on the Hausdorff measure. An analogous result for spinors was also covered in \cite{Baer1999zero}. This provides a further motivation for a more detailed investigation of the nodal sets of Dirac equations. 

\subsubsection*{Supersymmetric nonlinear sigma models}

A large class of interesting Dirac equations arise from the supersymmetric quantum field theory. 
For instance, in a two-dimensional supersymmetric nonlinear sigma model with gravitino, the fermionic field has to satisfy a critical nonlinear Dirac equation(see e.g.~\cite{Branding2015some,JKTWZ2018regularity} and references therein) of the form
\begin{align}\label{eq:SigmaModel}
 \D\psi= R(\phi, \psi,\psi)\psi \,.
\end{align}
Here~$\psi$ appearing twice in~$R(\phi,\psi,\psi)$ indicates that it is quadratic in~$\psi$.  
The above equation couples a spinor described by a Dirac equation with another elliptic equation in the map $\phi$ between manifolds.   
Moreover, the super-current, which is a conserved super-quantity, involves the solution spinor:
\begin{align}
 J= \parenthesis{ 2\Abracket{\dd\phi(e_j), \gamma(e_j)\gamma(e_k)\psi}+ |\psi|^2 \gamma(e_j)\gamma(e_k)\chi^j} e_k
\end{align}
where~$\phi$ is a map between manifolds and~$\chi$ stands for some auxiliary field, while~$(e_j)$ denotes a local orthonormal frame, see~\cite{JKTWZ2018symmetry}. 
In application, it is also desirable for the solution to have small nodal set, ideally isolated zeros.

\subsubsection*{Seiberg-Witten equations}
Another type of coupled equations involving Dirac operators comes from the Seiberg-Witten theory in dimension 3 and 4. In \cite{Taubes1996SW} pseudoholomorphic curves are obtained from the nodal set of harmonic spinors on symplectic 4-manifolds. This motivates the study of more detailed properties of the nodal set of some special class of spinors. In this model the spinor is coupled to (it is the so-called \emph{superpartner}) of a suitable gauge field. Geometrically, the latter corresponds to a connection $A$ on the spinor bundle. 
Depending on the geometry, the equation for the connection may take different forms, while the spinors are required to be harmonic:
\begin{align}
 \D_{g,A}\psi=0 \,.
\end{align}
The zero loci of the spinors, or more generally those of $\mathbb{Z}_2$-harmonic spinors, were analyzed in depth already by~\cite{DoanWalpuski2021existence, HaydysWalpuski2015compactness,Taubes2014zero} etc.

\subsubsection*{Dirac system with scalar potentials}
A particular case of \eqref{eq:Dirac eq-V} is given by
\begin{align}\label{eq:Dirac eq-f}
 \D\psi= f(\psi)\psi
\end{align}
where~$f \colon \Sigma_g M \to \R$ is a continuous function. 
This covers, for example, the already mentioned spinorial Yamabe problem \eqref{eq:spinYamabe} and the eigenspinor equation, as well as, for instance, the super Liouville systems \cite{JMW2020existence, JMW2021existence}. The case where $f$ does not depend on $\psi$, but it is only an endomorphism-valued  function on the manifold, is included in this case.
\smallskip

As already remarked, we point out that more general semilinear elliptic \emph{smooth} systems of first order are treated in \cite{Baer1999zero}, which also provides further applications.
\medskip

Notice that, in the above examples, the eigenspinor equation, the equations in supersymmetric nonlinear sigma model, and the Seiberg-Witten equation have a smooth structure, and then the dimension estimate for the nodal set follows by B\"ar's results \cite{Baer1997nodal, Baer1999zero}. In that case we prove a stratification result in Theorem \ref{thm:Stratification}, as well as for equations \eqref{eq:Dirac eq-V}, assuming that $V(\psi)$ satisfies some growth conditions near the zero section of the bundle. On the other hand, dimension estimates are proved under minimal assumptions in Theorem \ref{thm:DimEstimate}.


\subsection{Main results}

Let~$(M^n,g)$ be a complete smooth connected \emph{spin} Riemannian manifold of dimension~$n\ge 2$, with spinor bundle~$\Sigma_g M$. 
The Levi-Civita connection on tangent bundle induces a canonical spin connection~$\snabla^g$ or simply~$\snabla$ if there is no confusion. 
As a bundle of modules over the Clifford algebra bundle~$\Cl(M,g)$, there is a Clifford multiplication of tangent vectors on the spinor bundle, denoted by 
\begin{align}
 \gamma \colon TM \to \End(\Sigma_g M),\quad X\mapsto \gamma(X)
\end{align}
which satisfies the Clifford relation 
\begin{align}
 \gamma(X)\gamma(Y)+\gamma(Y)\gamma(X)= -2g(X,Y)\id_{\Sigma_g M}, \quad \forall X, Y \in \Gamma(TM). 
\end{align}
Via the Riemannian metric~$g$ we can identify the tangent bundle with the cotangent bundle, so that a cotangent vector, or even an exterior differential form, can be multiplied to a spinor. 
In particular, the composition of the following operators 
\begin{center}
 \begin{tikzcd}
  & \Gamma(\Sigma_g M) \arrow[r, "\snabla^g"]  
      & \Gamma(T^*M \otimes \Sigma_g M) \arrow[r, "g"]
        & \Gamma(TM\otimes \Sigma_g M)  \arrow[r, "\gamma"] 
           &\Gamma(\Sigma_g M)
 \end{tikzcd}
\end{center}
is the well-known Dirac operator, which is given in a local orthonormal frame~$(e_j)_{1\leq j\leq n}$ by
\begin{align}
 \D_g\psi=\gamma\circ g\circ \snabla\psi= \gamma\circ g \parenthesis{ \sum_{j=1}^n e^j\otimes \snabla_{e_j}\psi}
 =\gamma\parenthesis{ \sum_{j=1}^n e_j\otimes \snabla_{e_j}\psi}
 =\sum_{j=1}^n \gamma(e_j)\snabla_{e_j}\psi. 
\end{align}

Let~$V\colon \Sigma_g M \to \Sigma_g M$ be a fiber-preserving map which is locally Lipschtiz near the zero section and \emph{respects the zero section}, i.e.~$V(0)=0$.
In particular, for~$|\psi|$ small, 
\begin{align}
 |\D\psi|=|V(\psi)|\le \CLip |\psi|. 
\end{align}
The solutions of~\eqref{eq:Dirac eq-V} is of class~$C^{1,\alpha}$ for any~$\alpha\in (0,1)$ by the regularity theory of Dirac equations, see e.g.~\cite{Ammann2003Habil}. 
Hence, near a zero point~$x_0$ of a solution~$\psi$, we have 
\begin{align}
 |\D\psi(x)|\le \CLip |\psi| \le \CLip \dist(x, x_0)^{1+\alpha}.
\end{align}
Actually in applications we usually encounter the case that~$V(\psi)\precsim |\psi|^{p}$ for some~$p\ge 1$. 
\footnote{Most of our argument will not work for~$p\in (0,1)$. This interesting case will be addressed in a future work.}

For a solution~$\psi$ of~\eqref{eq:Dirac eq-V}, we are concerned with the zero sets
\begin{align}\label{eq:nodalset}
 \mathcal{Z}(\psi)\coloneqq \braces{ x\in M\mid \psi(x)=0}
\end{align}
and the singular sets 
\begin{align}
 \mathcal{S}(\psi)\coloneqq \braces{x\in M\mid \psi(x)=0, \; \; \snabla\psi(x)=0}.  
\end{align}
We are now in a position to state the first main result of the present paper.

\begin{thm}\label{thm:DimEstimate}
Let $V\colon \Sigma_g M \to \Sigma_g M$ be a fiber-preserving map which is locally Lipschitz near the zero section and respects the zero section. 
Then the nodal set \eqref{eq:nodalset} of a nontrivial solution $\psi\in C^{1,\alpha}(\Sigma^g M)$, $\alpha\in (0,1)$, to \eqref{eq:Dirac eq-V} has Hausdorff dimension less than or equal to~$n-2$.
\end{thm}
\begin{rmk}
In the above theorem (and in the other main results) we deal with $C^{1,\alpha}$ solutions to \eqref{eq:Dirac eq-V} as, for instance, we get that regularity if we start from a weak solution in the Sobolev space $H^{1/2}$, thanks to Schauder estimates.
\end{rmk}
\begin{rmk}
As pointed out in \cite[Sec. 3]{Baer1999zero}, the dimension estimate in \ref{thm:DimEstimate} is sharp. 
\end{rmk}
As already remarked, the above result generalizes the analogous statements in \cite{Baer1997nodal, Baer1999zero}, proved in the smooth case, and in \cite{BMW2021ground}, where the model nonlinearity $V(\psi)=\vert\psi\vert^{2/(n-1)}\psi$ on the round $n$-sphere is considered. Working with limited regularity one cannot, roughly speaking, take smooth expansions and argue correspondingly. This is done in B\"ar's papers, where a version of the Malgrange preparation theorem is used. 
As in \cite{BMW2021ground}, our argument relies on a local expansion of the Green function the Dirac operator; in particular, the Euclidean ones are sufficient for our purposes, see Section~\ref{sec:localexpansion}. 
We also exploit some ideas from the scalar case, contained in \cite{CaffarelliFriedman1979free, CaffarelliFriedman1985partial}.

\medskip

In order to proceed further, we need to impose some constraints on the map $V$. We consider the following cases, where we say that
\begin{itemize}
 \item $V$ is of type (Vf) if 
        \begin{align}\label{eq:Vf} \tag{Vf}
         V(\psi)= f(\psi)\psi
        \end{align}
        for some function $f\colon \Sigma_g M\to \R$ satisfying 
        \begin{align}\label{eq:f1} \tag{DV-f}
         |\D V(\psi)| \le C_f (|\psi|+ |\psi|^\tau|\snabla\psi|),  \qquad \mbox{ for some } \tau= \tau(f) \in (0,1);
        \end{align}

 \item $V$ is of type (Vq) if 
        \begin{align}\label{eq:Vq}\tag{Vq}
         |\D V(\psi)| \le \Cq \parenthesis{|\psi| + |\psi||\snabla \psi|};
        \end{align}
         

\end{itemize}

Observe that in the proofs we have to treat the above cases separately when dealing with the frequency function. 
Note that the \eqref{eq:Vq} case allows~$V$ to be locally (near the zero section) quadratic or super-quadratic, namely~$V(\psi)\precsim|\psi|^\sigma$ for some~$\sigma\ge 2$, and since~$\psi\in C^{1,\alpha}$,~$|\snabla\psi|$ is locally bounded, so that locally in a compact subset~$K$ we have 
\begin{align}
 |\D V(\psi)|\le \Cq (1+ \|\snabla\psi\|_{L^\infty(K)} ) |\psi|. 
\end{align}
This will be used later. 
When the potential part is of scalar type \eqref{eq:Vf} the right hand side can be sub-quadratic. 
It it worth noting that the case \eqref{eq:Vf} allows $f$ to be a function defined on~$M$ and independent of the spinor fiber, namely 
\begin{align}
 \D\psi = h \psi
\end{align}
for some~$h\in C^1(M;\R)$.
This takes care of~\eqref{eq:H-Dirac}, and when~$h=\lambda\in \R$ is a constant, this reduces to the eigenspinor equation.

\

In these cases, the result in Theorem \ref{thm:DimEstimate} can be improved. 
Indeed the zero set can be stratified according to the symmetries of the homogeneous blowup, as described in Section \ref{sec:stratification}. This kind of result is well-established for scalar functions satisfying elliptic second order equations,for which we refere to reader to e.g.~\cite{GarofaloLin1986monotonicity,Han1994singular}.
\medskip

\textbf{From now on we assume the manifold $M$ to be compact, for simplicity}.
\medskip 

Indeed, some of the quantities involved in the proofs of the next results rely on the geometry of the manifold and they are finite, for instance, if the manifold is compact. An example is given by the constants $C_{\textrm{am}}, r_{\textrm{am}}>0$ appearing in the statement of Theorem \ref{thm:AlmostMonotonicity}. More generally, the results below also applies when $(M,g)$ is of bounded geometry, that is, when $g$ is complete and the Riemann curvature tensor and all its covariant derivatives are bounded. This can be checked by inspection of the proofs.

\medskip

Observe that, since a priori the solution is not guaranteed to be smooth, we cannot define the vanishing order at a point by checking the lowest order of non-vanishing derivatives. 
We rather do so via a variant of the so-called \emph{frequency function}:
\begin{equation}\label{eq:freq}
 N(x,r) = \frac{r \int_{\p B_r(x)} \Abracket{\snabla_\nu \psi, \psi}\ds_r}{\int_{\p B_r(x)} |\psi|^2 \ds_r}\,,
\end{equation}
defined for $r>0$ small, where~$\p B_r(x)$ denotes the geodesic sphere in normal coordinates centered at~$x\in \mathcal{Z}(\psi)$. 
For harmonic functions, this corresponds to the frequency function used by Almgren in \cite{Almgren1979Dirichlet}  to study minimal submanifolds. The choice of the particular form of frequency function~$N(x,r)$ is partially inspired from \cite{GarofaloLin1986monotonicity}, meanwhile there are other formulation of frequency functions for spinors solving various Dirac-type equations, see~\cite{DoanWalpuski2021existence, HaydysWalpuski2015compactness, Taubes2014zero} and the references therein. 
The most significant property of the frequency function \eqref{eq:freq}, for our purposes, is the contained in the following result.

\begin{thm}\label{thm:AlmostMonotonicity}
Let $\psi\in C^{1,\alpha}$ be a solution to \eqref{eq:Dirac eq-V}, with $V$ as in \ref{thm:DimEstimate} satisfying \eqref{eq:Vf} or \eqref{eq:Vq}. For any~$C_N>0$, there exist~$\beta\in (0,1)$,~$ \rAM>0$ and~$\CAM>0$, all depending only on the geometry of~$M$ and on $V$ but independent of the solution~$\psi$, such that the function 
 \begin{align}
  (0, \rAM)\ni s\mapsto \exp\parenthesis{\frac{\CAM}{\beta+1} s^{\beta+1}} (N(x,s)+ C_N) 
 \end{align}
 is non-decreasing. 
\end{thm}

This is generally know as the \emph{almost monotonicty} of the frequency function, which is the reason for the subscript in $\CAM$.
Here it is proved for a non-homogenous equation and on a general manifold.  
For harmonic functions and harmonic spinors in Euclidean domains the corresponding frequency function can be easily shown to be monotone increasing, without need to add the constant~$C_N$. The property for harmonic functions is well-known (see e.g. \cite{Almgren1979Dirichlet}), while the same fact for harmonic spinors on manifolds can be obtained arguing as in \cite{HaydysWalpuski2015compactness}, working with their choice of the frequency function. 

\

The almost monotonicity of $N$ allows to show that the vanishing order at any point~$x\in \mathcal{Z}(\psi)$ is well-defined and equals
\begin{align}
 N(x,0)\coloneqq \lim_{r\to 0+} N(x,r)=\lim_{r\to 0+}  \exp\parenthesis{\frac{\CAM}{\beta+1} s^{\beta+1}} (N(x,s)+ C_N)-C_N.
\end{align}
Moreover, for compact manifolds, we have 
\begin{prop}\label{prop:UniformBound}
 Let~$\psi$ be a solution of~\eqref{eq:Dirac eq-V} under the assumptions of \ref{thm:AlmostMonotonicity}. 
 Then the set 
 \begin{align}
  \braces{ N(x,0)\; \mid \; x\in M}
 \end{align}
 is uniformly bounded in~$\R_{\ge 0}$. 
\end{prop}

In particular, this helps us to get rid of the set of points where the solution vanishes to infinite order (denoted by $\mathcal{N}_{\infty}$ in~\cite{Baer1999zero}). This is a desirable property, as  we are dealing with the spin Dirac operator, which is a geometric elliptic operator. 
With a uniform bound on the vanishing orders at hand, starting from some ideas used for the scalar case in \cite{Han1994singular}, we can establish the following stratification result.

\begin{thm}\label{thm:Stratification}
Let $V\colon \Sigma_g M \to \Sigma_g M$ be a fiber-preserving map which is locally Lipschitz near the zero section and respects the zero section, and be of type \eqref{eq:Vf} or \eqref{eq:Vq}. 
Then the nodal set $\cZ(\psi)$ of a solution to \eqref{eq:Dirac eq-V} is countably $(n-2)$-rectifiable. More precisely, it can be decomposed as
\begin{equation}\label{eq:Zdec}
\cZ(\psi)=\bigcup^{n-2}_{j=0} \cZ^j(\psi)
\end{equation}
where each $\cZ^j(\psi)$ is on a countable union of $j$-dimensional $C^1$ graphs for $0\leq j\leq n-3$ and~$\cZ^{n-2}$ is on a countable union of $(n-2)$-dimensional $C^{1,\alpha'}$ manifolds, for some~$0<\alpha'<1$.
\end{thm}
Recall that a subset of~$\R^n$ is called countably~$k$-$C^l$-rectifiable if it is contained in a countable union of~$k$-dimensional submanifolds of class~$C^l$ (we follow the definition in~\cite{Baer1999zero}). 
\smallskip

We conclude this Introduction, observing that there are still a lot of open questions concerning the nodal sets of spinors. 
In lower dimension, say dimension two, three and four, the zero loci of spinor solutions are geometrically significant and have rather nice structural properties, see e.g.~\cite{HaydysWalpuski2015compactness,Taubes2014zero}. 
However, it is unclear how to provide a similar description in higher dimensions. 
Moreover, as remarked before, the local finiteness of~$(n-2)$-dimensional Hausdorff measure would be expected, as in the smooth case~\cite{Baer1999zero}, thus calling for an estimate of the Hausdorff density in the general case.
Those questions will be handled in a future work. 
\
\subsection{Organization of the paper}
In Section \ref{sec:prel} we first recall the well-known Bourguinon-Gauduchon trivialization of the spinor bundle, which is later used in local arguments involving spinor solutions. Then we establish some preliminary estimates, involving a Hardy-type inequality and Pohozaev formula, needed for the frequency estimates. 
Then, in Section \ref{sec:localexpansion} we locally decompose a solution into a vectorial harmonic polynomial plus a higher order term, and establish the dimension estimate for the nodal set in Section~\ref{sec:dim estimate}. 

In the second part, we first analyze the frequency function in Section \ref{sec:freq} and then use it to have an uniform control on the vanishing order of a solution. Combining those properties with the local decomposition mentioned above, we can prove the desired stratification result in Section \ref{sec:stratification}.

\smallskip

\subsection*{Acknowledgements.}
The authors are grateful to Q. Han for providing useful references and to A. Malchiodi for many insightful conversations during their time in Pisa. 
\smallskip

W. B. has been supported by MUR grant \emph{Dipartimento di Eccellenza 2023-2027} of Dipartimento di Matematica at Politecnico di Milano, by the GNAMPA 2023 Project \emph{Analisi spettrale per equazioni di Dirac e applicazioni}, and by the PRIN 2022 Project \emph{Nonlinear effective equations in presence of singularities}. R. W. was partially supported by SISSA, Trieste, Italy where part of the project was carried out. 

\subsection*{Data availability statement}
 Data sharing not applicable to this article as no datasets were generated or analysed during the current study.

\section{Preliminaries}\label{sec:prel}

In this section we collect some basic material, together with some techinical lemmas. 
We refer to~\cite{LawsonMichelsohn1989spin} for a more complete account on spin geometry, and to \cite{GilbargTrudinger1983elliptic, Ammann2003Habil} for the elliptic regularity theory for scalar second order elliptic equations and Dirac equations, respectively. 

We will briefly recall the Bourguinon-Gauduchon construction of a suitable local trivialization for spinor bundles, which is particularly useful for our purposes, and the Green function for Dirac operators.
Afterwards we introduce a Hardy type inequality and Pohozaev type identity for spinors, for which we also provide a proof. 

\subsection{The Bourguinon-Gauduchon trivialization}
In~\cite{BG1992spineurs}, Bourguinon and Gauduchon analyzed the variation of the Dirac operator with respect to the Riemannian metrics. 
For that purpose a suitable trivialization for different metrics has to be chosen, which is quite convenient in application and now bears their names, here abbreviated as~\emph{BG-trivialization}.
We briefly recall the construction below, referring to \cite{BG1992spineurs, Maier1997generic, AmmannHumbert2005positive, AHM2006mass} for more details.

\

Fix~$x_0\in M$ and take a contractible neighborhood~$U$ of~$x_0$ on which the spinor bundle is trivial~$\Sigma_g M|_U \cong U\times \C^N$.
Let~$(x^1,\cdots, x^n)$ be the local normal coordinates with origin at~$x_0 = 0$. 
Then the coefficients of the Riemannian metric~$g= g_{ij}\dd x^i \dd x^j$ has the following local expansion 
\begin{align}
 g_{ij}(x)
 = \delta_{ij} +\frac{1}{3}R_{iklj}x^k x^l 
   +\frac{1}{6} R_{iklj,q} x^kx^j x^q + \cO(r^4) 
\end{align}
where~$r=|x|$ denotes the geodesic distance from the origin~$x_0=0$, see~\cite{LeeParker1987Yamabe}. 
Let~$g_0$ denote the Euclidean metric on~$U$ which is identified with a neighborhood of~$0\in\R^n$ via the coordinate chart, i.e.~$g_{0,ij}=\delta_{ij}$ in~$U$. 
Then for each~$x\in U$, there exists a unique positive definite symmetric matrix~$b(x)= (b_i^j(x))_{n\times n}$
such that 
\begin{align}
 b_k^i(x) g_{ij}(x) b^j_l(x) = \delta_{kl}. 
\end{align}
In matrix form, $b(x)$ is just the positive definite square root of~$(g_{ij})^{-1}$. 
Note that~$b(x)$ depends smoothly on~$x$ , and it has the local expansion
\begin{align}
 b_i^j(x)
 = \delta_i^j - \frac{1}{6} R_{iklj} x^k x^l -\frac{1}{12} R_{iklj,q} x^k x^l x^q + \cO(r^4), 
\end{align}
see e.g.~\cite{AGHM2008spinorial}. 

The coordinate tangent frame~$\braces{\p_i}_{1\le i\le n}$, with $\p_i\equiv \frac{\p}{\p x^i}$ for~$1\le i\le n$, is oriented but not orthonormal w.r.t.~$g$, though it is ortonormal w.r.t. the Euclidean metric~$g_0$.  
The matrix~$b(x)$ transforms it into a local oriented~$g$-orthonormal frame:
\begin{align}
 e_i(x) = b_i^j(x) \frac{\p}{\p x^j}\big|_x, \qquad 1\le i\le n.
\end{align}
Thus we have an isometry of the tangent bundles 
\begin{align}
 b\colon (TU, g_0) \to (TU, g). 
\end{align}
In other words,~$b$ is a local isomorphism of the principal~$\SO(n)$ bundles 
\begin{align}
 b\colon P_{\SO}(U, g_0) \to P_{\SO}(U,g). 
\end{align}
It lifts to an isomorphism of the corresponding principal~$\Spin(n)$ bundle
\begin{align}
 \tilde{b}\colon P_{\Spin}(U, g_0) \to P_{\Spin} (U,g).
\end{align}
and hence also induces an isometry of the associated spinor bundles
\begin{align}
 \beta\colon \Sigma_{g_0}U \to \Sigma_g U. 
\end{align}
Thus any spinor~$\psi \in \Gamma(\Sigma_g M)$ can locally (in~$U$) be expressed by an Euclidean spinor~$\varphi\in \Gamma(\Sigma_{g_0} U)$, which is nothing but a vector valued function 
\begin{align}
 \varphi\colon U\to \C^N, 
\end{align}
so that~$\psi= \beta(\varphi)$. 

We remark that, since the local spin geometry depends on the metric in an intrinsic way, the spin connections of the two spinor bundles are different, hence~$\beta$ is not an isomorphism of Dirac bundles; in particular,~$\beta$ does not preserve the Dirac operators.
Denoting by~$\gamma_{g_0}$ and~$\gamma_g$ the corresponding Clifford multiplications, we can choose the Clifford maps in such a way that 
\begin{align}
 \gamma_g(X)\beta(\varphi)=\beta(\gamma_{g_0}(X)\varphi), \qquad \forall X\in\Gamma(TU),\quad \forall \varphi\in \Gamma(\Sigma_{g_0}U). 
\end{align}
Then it was shown in~\cite{AGHM2008spinorial} that, for any~$\varphi \in \Gamma_{g_0} U$, 
\begin{align}
 \D_g \beta(\varphi)
 = \beta(\D_{g_0}\varphi)
   +\sum_{i,j} (b_i^j -\delta_i^j) \beta\parenthesis{\gamma_{g_0}(\p_i)\snabla^0_{\p_j}\varphi}
   +\frac{1}{4}\sum_{i,j,k}  \Gamma^k_{ij} \gamma_g(e_i)\gamma_g(e_j)\gamma_g (e_k)\beta(\varphi).  
\end{align}
Here~$\snabla^0$ is an abbreviation for the Euclidean connection~$\snabla^{g_0}$ on~$\Sigma_{g_0} U$ (and it is just~$\nabla=\dd$ componentwise since the Euclidean metric has zero connection forms), and the $\Gamma_{ij}^k$ are the Christoffel symbols for~$g$ in the orthonormal frame~$(e_1,\cdots, e_n)$:
\[
\nabla_{e_i} e_j = \Gamma_{ij}^k e_k\,.
\]
Note that for~$r=|x|$ small, 
\begin{align}
 \Gamma_{ij}^k(x)= \cO(r).
\end{align}
Returning to the solution~$\psi$ of~\eqref{eq:Dirac eq-V}, writing~$\psi=\beta(\varphi)$ and noting that\footnotetext{Here the subscript of the norms are indicating that the inner products on the spinor bundles are induced by the corresponding Riemannian metrics on the tangent bundle; strictly speaking we should write~$|\psi|_{g_s} = |\varphi|_{g_{0s}}$ where~$g_s$ stands for the inner product structure on~$\Sigma_g U$, and~$g_{0s}$ for~$\Sigma_{g_0}U$ respectively.}~$|\psi|_g = |\varphi|_{g_0}$, we see that 
\begin{align}
 \D_g\psi 
 = & \beta(\D_{g_0}\varphi)
    +\sum_{i,j} (b_i^j -\delta_i^j) \beta\parenthesis{\gamma_{g_0}(\p_i)\snabla^0_{\p_j}\varphi}
   +\frac{1}{4}\sum_{i,j,k}  \Gamma^k_{ij}\beta\parenthesis{ \gamma_{g_0}(\p_i)\gamma_{g_0}(\p_j)\gamma_{g_0} (\p_k)\varphi} \\
 = & \beta^{-1}\parenthesis{V(\beta(\varphi))},
\end{align}
that is,
\begin{align}\label{eq:EuclideanDirac-V}
 \D_{g_0}\varphi
 +\sum_{i,j} (b_i^j -\delta_i^j) \gamma_{g_0}(\p_i)\snabla^0_{\p_j}\varphi
   +\frac{1}{4}\sum_{i,j,k}  \Gamma^k_{ij} \gamma_{g_0}(\p_i)\gamma_{g_0}(\p_j)\gamma_{g_0} (\p_k)\varphi
  = \widetilde{V}(\varphi)
\end{align}
where 
\begin{align}
 \widetilde{V}(\varphi)\coloneqq \beta^{-1}\circ V\circ \beta(\varphi) .
\end{align}
Since~$\beta$ is isometric, the new map $\widetilde{V}$ is of the same type as~$V$: this is clear for~$V$ of type (Vf), while for~$V$ of type (Vq), note that for any~$X\in\Gamma(TM)$,
\begin{align}
 \snabla^{g_0}_X \parenthesis{\beta^{-1}\circ V\circ \beta(\varphi) }
 =\snabla^{g_0}_X (\beta(V(\psi)))
 = (\nabla_X\beta)V(\psi)+ \beta\parenthesis{ \snabla_X V(\psi)}
\end{align}
 so that 
 \begin{align}
  \left| \D_{g_0} \parenthesis{\beta^{-1}\circ V\circ \beta(\varphi) }\right|
  \le |\nabla \beta| |V(\psi)| + |\D V(\psi)|
 \end{align}
 which satisfies a growth condition of the form~\eqref{eq:Vq} locally near the zero section of~$\Sigma_g M$.

Moreover, the Christoffel symbols~$\Gamma^k_{ij}$ with respect to~$(g, (e_i))$ relates to the coordinates Christoffel symbols~$\mathring{\Gamma}^k_{ij}$ defined by~$\nabla_{\p_i}\p_j = { \mathring{\Gamma}^k_{ij}}\p_k$ via (see~\cite{AGHM2008spinorial})
\begin{align}
 \Gamma^k_{ij}=\parenthesis{ b_i^p (\p_p b_j^l) + b_i^p b_j^q \mathring{\Gamma}^l_{pq}} (b^{-1})^k_l.
\end{align}

Note that the nodal sets of~$\psi$ and of~$\varphi$ are the same, we can thus exchangeably consider the solution of the equation~\eqref{eq:EuclideanDirac-V} in the sequel. 
However, since the second term in the equation also involves first order derivatives, which still cause some troubles for local computations, so we will also use the formulation in terms of the original Riemannian metric~$g$ whenever convenient.

\subsection{Fundamental solution of Euclidean Dirac operator}

Recall that the fundamental solutions~$\G(x,y)$ of the Dirac operator~$\D_g$ is a section of the bundle~$\pi_1^*\Sigma_g M \otimes (\pi_2^*\Sigma_g M)^*\to M\times M \setminus \diag(M\times M) $ which is singular along the diagonal
\begin{align}
 \diag(M\times M) \coloneqq \braces{(x,x)\in M\times M\mid x\in M}
\end{align}
where~$\pi_1, \pi_2$ denote the canonical projections to the first and second factors, respectively. 
It satisfies 
\begin{align}
 \D_{g,x} \G(x,y) = \delta_y(x) \id_{\Sigma_y M}
\end{align}
in the sense of distributions: for any~$\psi_1, \psi_2\in \Gamma(\Sigma_g M)$, 
\begin{align}\label{eq:Fundamental-Dirac}
 \int_M \Abracket{ \G(x,y)\psi_1(y), \D_g^*\psi_2 (x)}\dv_g (x) 
 = \Abracket{\psi_1(y),\psi_2(y)}, \qquad \forall y\in M. 
\end{align}

In the Euclidean domain~$(U, g_0)\subset \R^n$, the fundamental solution of~$\D_{g_0}$ is given by 
\begin{align}
 \G_0(x,y)=\D_{g_0,x}G_0(x,y)= -\frac{1}{n\omega_n}\frac{\gamma(x-y)}{|x-y|^n}
\end{align}
where~$G_0(x,y)$ denotes the fundamental solution of~$-\Delta_{g_0}$ satisfying~$-\Delta_{g_0} G_0(x,y)= \delta_y(x)$; for~$n\ge 3$ 
\begin{align}
 G_0(x,y)= \frac{1}{(n-2)n\omega_n}\frac{1}{|x-y|^{n-2}}\; ,
\end{align}
while for~$n=2$ 
\begin{align}
 G_0(x,y)=\frac{1}{2\pi}\log\frac{1}{|x-y|}\; . 
\end{align}

This is due to the fact 
\begin{align}
 \D_{g_0}^2 = -\Delta_{g_0} \id_{N\times N}=\diag(-\Delta_{g_0},\cdots, -\Delta_{g_0}).
\end{align}
It is readily seen that~$\D_{g_0,x} \G_0(x,y)=\delta_y(x)\id_{\Sigma_{g_0, y}}U$.
Then for Euclidean spinors~$\varphi_1,\varphi_2$ defined in~$U$, integration by parts in~\eqref{eq:Fundamental-Dirac} gives
\begin{align}
 \int_{U} & \Abracket{ \G_0(x,y)\varphi_1(y), \D_{g_0,x}^*\varphi_2 (x)}\dd x \\
  &= \int_{U} \Abracket{ \D_{g_0, x}\G_0(x,y)\varphi_1(y), \varphi_2(x)}\dd x 
      +\int_{\p U} \Abracket{\G_0(x,y)\varphi_1(y), \gamma(\nu)\varphi_2(x)}\ds_x \\
  &=\Abracket{\varphi_1(y), \varphi_2(y)}_{\Sigma_{g_0, y} U}
    +\int_{\p U} \Abracket{\varphi_1(y), \G_0(y,x)\gamma(\nu)\varphi_2(x)} \ds_x
\end{align}
where in the last step we used the skew-symmetry of Clifford multiplication of tangent vectors, and that~$G_{0}(x,y)= - G_0(y,x)$. 
That~$G(x,y)$ is anti-symmetric and self-adjoint actually holds ture in general case, see~\cite{AHM2006mass}. 
In the above computation, only~$\varphi_1(y)$ is needed, and~$\varphi_1(y)$ can be arbitrary in~$\Sigma_{g_0, y}U$.
Moreover,~$\D_{g_0}^*=\D_{g_0}$. 
Thus we obtain the spinorial Newton represenation formula 
\begin{align}\label{eq:spinNewton}
 \varphi(y)= \int_{U} \G_0(y,x)\D_{g_0,x}\varphi(x)\dd x
          -\int_{\p U} \G_0(y,x)\gamma_{g_0}(\nu(x))\varphi(x)\ds_x. 
\end{align} 
This will be the starting point for the proof of the local expansion provided in Section \ref{sec:localexpansion}.

\subsection{A Hardy type inequality}

 We need a general form of the Hardy inequality with boundary terms in the Riemannian case.
 For completeness we include a proof here. 
 
 Let~$B_r\equiv B_r(0)\subset U$ be a geodesic ball centered at~$x_0$ in~$(M,g)$. 
 The Riemannian metric in normal coordinates~$(\rho, \theta^1,\cdots, \theta^{n-1})$ takes the form 
 \begin{align}
  g = \dd \rho^2 + \rho^2 w_{ij}(\rho,\theta)\dd\theta^i\dd\theta^j. 
 \end{align}
 The corresponding volume form is 
 \begin{align}
  \dv_g 
  =& \sqrt{\det(g)}\dd\rho\dd\theta^1\cdots \dd\theta^{n-1}
  = \rho^{n-1}\sqrt{\det(w_{ij}(\rho,\theta))}\dd\rho\dd\theta
 \end{align}
 where we have abbreviated~$\dd\theta\equiv \dd\theta^1\cdots \dd\theta^{n-1}$.
 \begin{rmk}
  For spaceforms, we know that 
  \begin{align}
   w_{ij}(0,\theta)=\omega_{ij}, & & \frac{\p}{\p\rho} w_{ij}(\rho,\theta)|_{\rho=0} =0,
  \end{align}
  while the second derivatives of~$w_{ij}(\rho,\theta)$ would reflect curvature information. 
  Here~$\omega_{ij}$'s are the metric coefficients for the standard Euclidean spherical coordinates. 
 \end{rmk}
 
 For general~$g$, for~$\rho>0$ small, we may assume that 
 \begin{align}
  w_{ij}(\rho,\theta)= \omega_{ij}+ \cO(\rho^2), 
  & & 
  \frac{\p}{\p\rho} w_{ij}(\rho,\theta) = \cO(\rho). 
 \end{align}
 In particular, 
 \begin{align}
  \det(w_{ij}(\rho,\theta))=1+ \cO(\rho^2), & &  
  \frac{\p}{\p\rho}\det(w_{ij}(\rho,\theta))
  = \cO(\rho).
 \end{align}
 In the sequel we will encounter the term 
 \begin{align}
 W(\rho,\theta)
 \coloneqq &\frac{1}{\sqrt{\det(w_{ij}(\rho,\theta))}}\frac{\p}{\p \rho}\sqrt{\det(w_{ij}(\rho,\theta))}
 =\frac{1}{2}\frac{1}{\det(w_{ij}(\rho,\theta))}\frac{\p}{\p \rho}\det(w_{ij}(\rho,\theta)) \\
 =&\frac{1}{2}\frac{\p}{\p \rho}\ln\det(w_{ij}(\rho,\theta))
\end{align}
which can be estimated by 
\begin{align}
\left|W(\rho,\theta) \right|
 \le \CCoord \rho, \qquad \mbox{ for } 0<\rho < \rCoord
\end{align}
for some~$\rCoord>0$ small and~$\CCoord=\CCoord(n,g)>0$. 
Actually the constant~$\CCoord$ depends only on the dimension and local curvature, which can be chosen uniformly on compact regions of~$M$, or for regions with bounded geometry; in particular this is the case for smooth and compact~$M$. 
Though the following Hardy inequality for scalar functions on bounded domain is well-known, we present a proof to highlight the universality of the constants. 

 \begin{lemma}
  There exists~$\rHardy= \rHardy(n, g)>0$ such that for any~$r<\rHardy$ and any~$u\in W^{1,2}(B_r)$, there holds 
  \begin{align}\label{eq:Hardy-scalar}
   \int_{B_r} \frac{u^2}{\rho^2} \dv_g 
   \le  \frac{2}{n-1}\frac{1}{r}\int_{\p B_r} u^2\ds_r 
    + \frac{4}{(n-2)^2}\int_{B_r}|\nabla u|^2\dv_g . 
  \end{align}

 \end{lemma}
 \begin{proof}
  In normal spherical coordinates and for~$r<\rCoord$, we have 
  \[
  \begin{split}
   \int_{B_r} \frac{u^2}{\rho^2}\dv_g 
   =&\int_{\sph^{n-1}} \dd\theta \int_0^r \frac{u^2}{\rho^2} \rho^{n-1} \sqrt{\det(w_{ij}(\rho,\theta))}\dd\rho\\ 
   =&\frac{1}{n-2}\int_{\sph^{n-1}}\dd\theta\int_0^r u^2 \sqrt{\det(w_{ij})}\dd\rho^{n-2} \\
   =&\frac{1}{n-2}\int_{\sph^{n-1}} u^2 \sqrt{\det(w_{ij})}\rho^{n-2}\dd\theta\big|_{0}^r \\
    &\qquad -\frac{1}{n-2}\int_{\sph^{n-1}} \int_0^r 2u\frac{\p u}{\p \rho}\sqrt{\det(w_{ij})}\rho^{n-2}\dd\rho \dd\theta \\
     & \qquad -\frac{1}{n-2}\int_{\sph^{n-1}}\int_0^r u^2 \frac{\p}{\p \rho}\sqrt{\det(w_{ij})} \rho^{n-2}\dd\rho\dd\theta \\
   =&\frac{1}{n-2}\int_{\p B_r} \frac{u^2}{r} \ds_r 
    -\frac{1}{n-2}\int_{B_r} 2\frac{u}{\rho} u_\rho \dv_g\\
      &\qquad \qquad \qquad \qquad 
    -\frac{1}{n-2}\int_{B_r} \frac{u^2}{\rho}\frac{\p_\rho \sqrt{\det(w_{ij})}}{\sqrt{\det(w_{ij})}} \dv_g \\
   \le& \frac{1}{n-2}\frac{1}{r}\int_{\p B_r} u^2 \ds_r
       +\frac{\eps}{n-2}\int_{B_r}\frac{u^2}{\rho^2}\dv_g+\frac{1}{(n-2)\eps}\int_{B_r} |\nabla u|^2 \dv_g 
       \\
      & +\frac{\CCoord r^2}{n-2}\int_{B_r} \frac{u^2}{\rho^2} \dv_g. 
 \end{split}
 \]
 We choose~$\rHardy>0$ such that~$\frac{\CCoord \rHardy^2}{n-2}<\frac{1}{4}$, and take $\eps=\frac{n-2}{4}$. 
 Then for~$r< \rHardy <\sqrt{\frac{n-2}{4 \CCoord}}$, we have  
 \begin{align}
  \int_{B_r}\frac{u^2}{\rho^2}\dv_g
  \le \frac{2}{n-2}\frac{1}{r}\int_{\p B_r} u^2\ds_r 
    + \frac{4}{(n-2)^2}\int_{B_r}|\nabla u|^2\dv_g 
 \end{align}

 \end{proof}
Once~$\CCoord$ is uniformly fixed on~$M$ (or on a compact subset), the~$\rHardy$ can be also uniformly chosen. As a result, \emph{the coefficients in the Hardy inequality above are only dimensional constants}. 
This is quite relevant for later applications. 
For later convenience we write~$\CHardy=\frac{2}{n-2}$. 
We apply the above Hardy type inequality to~$u=|\psi|$ and use the Kato's inequality to obtain 

\begin{lemma}
 For~$\rHardy(n,g)>0$ as above,~$r<\rHardy$ and~$\psi\in W^{1,2}(B_{\rHardy})$, we have 
 \begin{align}\label{eq:Hardy-spinor}
 \int_{B_r} |\psi|^2\dv_g 
 \le r^2\int_{ B_r}\frac{|\psi|^2}{\rho^2}\dv_g 
 \le \CHardy r\int_{\p B_r} |\psi|^2\ds_r 
    + \CHardy^2 r^2 \int_{B_r} |\snabla\psi|^2\dv_g. 
\end{align}
\end{lemma}

\subsection{A Pohozaev type formula for spinors}

Recall the Pohozaev identity for a scalar function~$u\in C^2(B_R)$ and~$r<R$
  \begin{align}
   -\int_{B_r} 2( \vec{x}\cdot\nabla u) \Delta u + (n-2)u\Delta u \dx 
   =\int_{\p B_r} r|\nabla u|^2-2r |\p_\nu u|^2 - (n-2)u\p_\nu u \ds .
  \end{align}
Here~$B_r=B_r(0)\subset\R^n $ is an Euclidean ball.

A similar formula holds for spinors on manifolds.
Indeed, let~$B_R$ denote a geodesic ball in normal coordinates, and write~$\vec{x}= x^j \frac{\p}{\p x^j}$ for the local position vector field.
Recall that for a general vector field~$X=X^j\frac{\p}{\p x^j}$, the divergence is given by
\begin{align}
 \diverg_g(X)= \frac{\p X^j}{\p x^j} + \mathring{\Gamma}^j_{jk}X^k 
\end{align}
with~$|\mathring{\Gamma}^i_{jk}(x)| = \cO(|x|)$.
In particular, for a spinor field~$\psi\in C^2(B_R)$ and~$r<R$, the vector field (here we identify a differential form part with a vector field via the Riemannian metric)
\begin{align}
 X = |\snabla\psi|^2 \vec{x}-2\Abracket{(\snabla_{\vec{x}}\psi),\snabla\psi} -(n-2)\Abracket{\psi,\snabla\psi}
\end{align}
has divergence (where we have dropped the Christoffel symbol terms in the computation, but this doesn't affect the conclusion since both the LHS and RHS are tensorial)
\begin{equation}
\begin{split}
    & \diverg_g\parenthesis{|\snabla\psi|^2 \vec{x}-2\Abracket{(\snabla_{\vec{x}}\psi),\snabla\psi} -(n-2)\Abracket{\psi,\snabla\psi}} \\
     =& e_i\parenthesis{|\snabla\psi|^2 x^i-2x^j\Abracket{\snabla_j\psi,\snabla_i\psi}-(n-2)\Abracket{\psi,\snabla_i\psi}} \\
     =& n|\snabla\psi|^2+2x^i\Abracket{\snabla_i\snabla_j\psi,\snabla_j\psi}  \\
      &-2\delta_i^j \Abracket{\snabla_j\psi,\snabla_i\psi}
       -2x^j\Abracket{\snabla_i\snabla_j\psi,\snabla_i\psi}
       -2x^j\Abracket{\snabla_j\psi,\snabla_i\snabla_i\psi} \\
      &-(n-2)\Abracket{\snabla_i\psi,\snabla_i\psi}
       -(n-2)\Abracket{\psi,\snabla_i\snabla_i\psi} \\
     =&2x^i\Abracket{\snabla_i\snabla_j\psi-\snabla_j\snabla_i\psi,\snabla_j\psi}
      +2x^j\Abracket{\snabla_j\psi, \snabla^*\snabla\psi }
      +(n-2)\Abracket{\psi,\snabla^*\snabla\psi} \\
     =&2\Abracket{R_{\vec{x},e_j}\psi,\snabla_j\psi}
       +\Abracket{2\snabla_{\vec{x}}\psi+(n-2)\psi, \snabla^*\snabla\psi}.
       \end{split}
    \end{equation}
By Stokes theorem we get
    \begin{multline}\label{eq:Pohozaev-spinor}
     \int_{\p B_r} r|\snabla\psi|^2-2r|\snabla_{\nu}\psi|^2-(n-2)\Abracket{\psi,\snabla_\nu\psi}\ds_r \\
     =\int_{B_r} 2\Abracket{R_{\vec{x},e_j}\psi,\snabla_j\psi}
       +\Abracket{2\snabla_{\vec{x}}\psi+(n-2)\psi, \snabla^*\snabla\psi}\dv_g.
    \end{multline}
Note that by assumption we only have~$\psi\in C^{1,\alpha}$ so in general~$\snabla^*\snabla\psi$ is not defined. 
But for solutions of~\eqref{eq:Dirac eq-V} with~$V$ being of the types under consideration, we have 
    \begin{align}
     \D^2 \psi = \D V(\psi)=\gamma(e_j)\snabla_{e_j} V(\psi) 
    \end{align}
where the RHS is pointwisely defined, so by Lichnerowicz formula 
    \begin{align}
     \snabla^*\snabla+\frac{\Scal}{4}=\D^2
    \end{align}
we can define~$\snabla^*\snabla \psi$ as well, and
\begin{align}
 |\snabla^*\snabla\psi| \le \frac{|\Scal|}{4}|\psi|+ |\D V(\psi)|.
\end{align}
Hence the above Pohozaev formula~\eqref{eq:Pohozaev-spinor} is also well-defined.

\section{Local expansion around a zero}\label{sec:localexpansion}

In this section we prove that a spinor~$\psi$ solving \eqref{eq:EuclideanDirac-V}, vanishing at~$x_0$ admits a local expansion in~$U(x_0)$, with principal term being a vector consisting of harmonic polynomials. 
Similar results are well established for scalare equations, see e.g. ~\cite{CaffarelliFriedman1979free, CaffarelliFriedman1985partial, Han1994singular}. 
This fact in the spinorial case was essentially proved in~\cite{BMW2021ground}, and we briefly sketch it below. 

As before we denote a solution of~\eqref{eq:Dirac eq-V} on~$(M,g)$ by~$\psi$, and denote the corresponding local Euclidean spinor defined in~$(U, g_0)$ by~$\varphi$; they are related by~$\psi=\beta(\varphi)$ in~$U$. 
The spinor~$\varphi$ satisfies the equation~\eqref{eq:EuclideanDirac-V}:
\begin{align}
 \D_{g_0}\varphi+\sum_{i,j}(b_i^j-\delta_i^j)\gamma_{g_0}(\p_i)\snabla^0_{\p_j}\varphi
 +\frac{1}{4}\sum_{i,j,k}\Gamma^k_{ij}\gamma_{g_0}(\p_i)\gamma_{g_0}(\p_j)\gamma_{g_0}(\p_k)\varphi
 =\widetilde{V}(\varphi)
\end{align}
with~$\widetilde{V}\colon \Sigma_{g_0}U \to \Sigma_{g_0}U$ being Lipschitz and of the same type as~$V$.

Recall that the solution~$\psi$ to~\eqref{eq:Dirac eq-V} has regularity at least~$C^{1,\alpha}$, and so is~$\varphi$ since~$\beta$ is a bundle isometry. 
In particular, in the normal coordinates centered at~$0=x_0\in \mathcal{Z}(\psi)$, we have 
\begin{align}
 |\varphi(x)|\le C|x|^{1+\alpha}, & & 
 |\snabla^0\varphi(x)|\le C|x|^{\alpha}, \qquad \mbox{ for any }  \alpha \in (0,1). 
\end{align} 
Moreover, due to the equation~\eqref{eq:EuclideanDirac-V} and the fact that~$b_i^j-\delta_i^j= \cO(|x|^2)$, ~$\Gamma^k_{ij}(x)=\cO(|x|)$ and~$\widetilde{V}(\varphi)\le \CLip |\varphi(x)|$, we see that
\begin{align}
 |\D_{g_0}\varphi| \le C|x|^{1+\alpha},\qquad \mbox{ for any } \alpha\in (0,1).  
\end{align}

With the aid of the representation formula~\eqref{eq:spinNewton}, we can get the following decomposition. 
\begin{lemma}\label{lem:decomp-spinor}
 Suppsoe~$\varphi\colon U\to \R^{2N}$ is~$C^{1,\alpha}$,~$\varphi(0)=0$ and satisfies the local growth condition
 \begin{align}
  |\D\varphi(x)| \le C_\sigma |x|^{\sigma}
 \end{align}
 for some non-integer~$\sigma>0$.
 Then there exists~$R\in (0,1)$ such that 
 \begin{align}\label{eq:decomp}
  \varphi(x)= P(x)+ Q(x) \qquad \mbox{ in } B_R(0)\subseteq U
 \end{align}
 for some~$P,Q\colon B_R \to \R^{2N}$ where the nonzero components of~$P$ are harmonic polynomials of degree~$[\sigma]+1$, and
 \begin{align}\label{eq:Q-growth}
  |Q(x)|\le C'_\sigma |x|^{\sigma+1}, & & 
  |\snabla^0 Q(x)| \le C''_\sigma |x|^\sigma,
  \qquad \mbox{ in } B_R(0).
 \end{align}
 Moreover,~$P$ is a~$\D_{g_0}$-harmonic spinor, i.e.~$\D_{g_0}P=0$. 

\end{lemma}

Note that initially we may start with~$\sigma=\alpha\in (0,1)$ and get a decomposition with $\deg(P)=1$.
But if~$\psi$ or equivalently~$\varphi$ vanishes at~$0=x_0$ to a higher order, a notion which will be clear later, we may get a larger~$\sigma$ and hence also a~$P$ of higher degree. 

\begin{proof}
 The proof is essentially given in~\cite{BMW2021ground}, thus here we only sketch it. 
 
 By~\eqref{eq:spinNewton}, we have
 \begin{align}
  \varphi(x)=\int_{B_R} \G_0(x,y)\D_{g_0} \varphi(y)\dd y 
   - \int_{\p B_R} \G_0(x,y)\gamma_{g_0}(\nu(y))\varphi(y)\ds_y.
 \end{align}
 The Green operator~$\G_0$ admits a local expansion in terms of homogeneous~$\D_{g_0}$-harmonic polynomials, i.e., Taylor expanding
 \begin{align}
  \G_0(x,y)
  =\sum_{k=0}^\infty \sum_{|\beta|=k}\p_x \G_0(-y)\frac{x^\beta}{\beta!}.
 \end{align}
 For each~$k\ge 0$, the~$k$-th summand above, 
 \begin{align}
  \sum_{|\beta|=k}\p_x \G_0(-y)\frac{x^\beta}{\beta!}\equiv \mathfrak{G}_k(x,y)
 \end{align}
 is~$\D_{g_0}$-harmonic in~$x$ and consists of degree~$k$ homogeneous polynomials. 
 
 Then we have 
 \begin{align}
  P(x)
  \coloneqq
  &\int_{B_R}\sum_{k=0}^{[\sigma]+1}\mathfrak{G}_k(x,y) \D_{g_0}\varphi(y)\dd y 
   +\int_{B_R \cap B_{(1+\frac{1}{\sigma})|x|}  }  \sum_{k\ge [\sigma]+2} \mathfrak{G}_k(x,y)\D \varphi(y)\dd y
   \\
   &+\int_{\p B_R} \sum_{k=0}^{[\sigma]+1}\mathfrak{G}_k(x,y)\gamma_{g_0}(\nu(y))\varphi(y)\ds_y,
 \end{align}
 which can be shown to be~$\D_{g_0}$-harmonic in~$x$, and whose components are homogeneous polynomials of degree at most~$[\sigma]+1$. 
 
 The other part
 \begin{align}
  Q(x)
  =& \int_{B_R \setminus B_{(1+\frac{1}{\sigma})|x|}} \sum_{k\ge [\sigma]+2} \mathfrak{G}_k(x,y)\D_{g_0}\varphi(y)\dd y  \\
  &+\int_{\p B_R} \sum_{k\ge [\sigma]+2} \mathfrak{G}_k (x,y)\gamma_{g_0}(\mu(y))\varphi(y)\ds_y
 \end{align}
is readily seen to satisfy~\eqref{eq:Q-growth}. 
For more details we refer to~\cite{BMW2021ground} and~\cite{CaffarelliFriedman1979free, CaffarelliFriedman1985partial}. 
\end{proof}

We remark that~$P(0)=0$, so the vector polynomial~$P$ has no degree zero part.
The lowest degree part in~$P$, is the leading term for the dimension estimates in next section. 
Therefore, we restate the result with $P$ homogeneous and~$\D_{g_0}$-harmonic.

\begin{prop}\label{prop:decomp-k}
 Let~$\varphi\in C^{1,\alpha}(U, \R^{2N})$ satisfies~$\varphi(0)=0$, ~$\varphi\not\equiv 0$ in~$U$, and 
 \begin{align}
  |\D_{g_0}\varphi|\le C|\psi|^{\bar{\sigma} }
 \end{align}
 for some~$C\ge 0$ and~$\bar{\sigma}\ge 1$.
 
 Then there exist~$k\ge 1$ and~$R\in (0,1)$ such that in the ball~$B_R(0)\subset U$ we have the decomposition 
 \begin{align}\label{eq:decomp-k}
  \varphi(x)= P_k(x)+ Q_k(x)
 \end{align}
 where~$P_k, Q_k\colon B_R \to \R^{2N}$ satisfy 
 \begin{itemize}
  \item $P_k\neq 0$, $\D_{g_0} P_k =0$, and the nonzero components of~$P_k$ are homogeneous harmonic polynomials of degree~$k$;
  \item $Q_k$ is continuous, $Q_k(0)=0$ and for any~$\delta \in (0,1)$ there exists a constant~$C(\delta)>0$ such that
        \begin{align}
         |Q_k(x)|\le C(\delta) |x|^{k+\delta}, & & 
         |\nabla Q_k(x)| \le C(\delta) |x|^{k+\delta-1}, \qquad \mbox{ in } B_R(0). 
         \end{align}

 \end{itemize}

\end{prop}
For the proof, given the regularity of $\varphi$, the previous results applies and then it suffices to collect the nonzero lowest degree part~$P_k$ of~$P$ in~\eqref{eq:decomp}, and put~$Q_k= (P- P_k ) + Q$. 
That~$P$ is nonzero (so that such a~$k$ exists) follows from the strong unique continuation principle for the standard Euclidean Dirac operator~$\D_{g_0}$.

\section{Dimension estimate of the nodal domain}\label{sec:dim estimate}
In this section we prove that the nodal set $\cZ(\psi)$ of a solution $\psi\in C^{1,\alpha}(\Sigma_g M)$ to \eqref{eq:Dirac eq-f} has Hausdorff dimension less than or equal to $n-2$, under the assumptions of Theorem \ref{thm:DimEstimate}. 
Note that under the bundle isomorphism~$\beta$ we have~$\cZ(\psi)=\cZ(\varphi)$ where~$\psi=\beta(\varphi)$, thus it suffices to consider the Euclidean spinor~$\varphi$.
To this aim we decompose the nodal set as follows
\begin{equation}\label{eq:Zsplit}
 \cZ(\psi)=\cZ_1(\psi)\cup\cZ_{\geq2}(\psi)\,,
\end{equation}
where $\cZ_1(\psi)$ is the set of points where $\psi$ vanishes to first order, that is, where the gradient~$\snabla\psi$ does not vanish,  while $\cZ_{\geq2}(\psi)$ denotes the set of points where both~$\psi$ and~$\snabla\psi$ vanish, namely the \emph{singular set} of~$\psi$,~$\mathcal{S}(\psi)$.

In terms of the decomposition~\eqref{eq:decomp-k},~$\cZ_1(\psi)=\cZ_1(\varphi)$ corresponds to those points where~$k=1$, while~$\cZ_{\ge 2}(\psi)=\cZ_{\ge 2}(\varphi)$ corresponds to those points where~$k\ge 2$.  

\medskip
     
In the sequel, for simplicity of notation, we assume that a solution $\psi=\beta(\varphi)\in C^{1,\alpha}$ is given and we drop the dependence on it in the notation, whenever there is no ambiuity. 
Thus we denote the nodal set by $\cZ$, and so on.
     
\medskip
     
We treat differently points in the nodal set, according to the splitting \eqref{eq:Zsplit}, and prove the desired dimension estimate for the set $\cZ_1$ first, and then dealing with $\cZ_{\geq 2}=\mathcal{S}$.

\subsection{Dimension estimates for \texorpdfstring{$\cZ_1$}{Z1} }

In contrast to scalar functions, whose regular level sets are smooth~$(n-1)$-dimensional hypersurfaces, we have the following result.
\begin{lemma}\label{lem:Z1estimate}
     The set $\cZ_1$ is $(n-2)$-$C^{1,\alpha}$-rectifiable. 
\end{lemma}
    \begin{proof}
Let $x_0\in\cZ_1$ and take a normal coordinate chart around~$x_0 =0$. 
It sufficies to show that $\cZ\cap B_r$ is contained in a~$C^{1,\alpha}$-rectifiable subset of dimension at most $n-2$ for some small~$r>0$. 

\smallskip

By assumptions, the decomposition \eqref{eq:decomp-k} holds around $x_0$, with $k=1$. Then we have $P_1=(P^1_1, \cdots, P^{2N}_1)$, each~$P^j_1$ being a homogenous polynomial of degree one, and the vector space~$\mathscr{P}=\Span_\R\{P^1_1, \cdots , P^{2N}_1\}$ is non-trivial. 
We claim that~$\mathscr{P}$ cannot be one-dimensional.

By contradiction, suppose that there exists a non-zero linear function $p(x^1, \cdots, x^n)$ and constants $c^1, \cdots, c^{2N}\in \R$ such that
\begin{equation}
 P^j_1=c^j p ,\quad 1\le j \le 2N
\end{equation}
where at least one coefficient $c^j$ is non-zero. 
By \eqref{eq:decomp}, we have $\nabla Q_1(0)=0$. 
Then at~$x_0=0$ we have
\begin{equation}
 \D_{g_0}\varphi(0)
 =\sum_{1\le\alpha\le n} \gamma_{g_0}(\p_\alpha)\nabla_{\p_\alpha}\varphi(0)
 =\sum_{1\le\alpha\le n} \gamma_{g_0}(\p_\alpha)\nabla_{\p_\alpha}P_1(0).
\end{equation}
Since~$p(x^1,\cdots, x^n)$ is linear, up to a linear transformation on~$B_1(0)\subset \R^n$, we may assume that
~$p(x^1,\cdots, x^n)= x^1$, and hence~$\nabla_{\p_\alpha}p=\delta_{1\alpha}$. 
Consequently, there holds
\begin{equation}
 \D_{g_0}\varphi(0)= \sum_{1\le\alpha\le n}\gamma_{g_0}(\p_\alpha)
 \begin{pmatrix}
  c^1 \\\vdots \\ c^{2N}
 \end{pmatrix}\delta_{1\alpha}
 =\gamma_{g_0}(\p_1)\begin{pmatrix}
  c^1 \\\vdots \\ c^{2N}
 \end{pmatrix}. 
\end{equation}
On the other hand, since~$V$ respects the zero section, $\D\psi(0)=V(0)=0$, hence also $\D_{g_0}\varphi(0)=0$. 
But since $\gamma_{g_0}(\p_1)$ in invertible, we get $c^1=\cdots=c^{2N}=0$, which is a contradiction. 

Therefore, the vector space~$\Span_\R\{P^1_1, \cdots , P^{2N}_1\}$ is at least two-dimensional.
We may suppose that $P^1_1,P^2_1$ are linearly independent, and then so are their gradients $\nabla P^1_1,\nabla P^2_1$. 
Note that
\[
\cZ\cap B_\rho=\{x\in B_\rho :\varphi(x)=0 \}\subseteq\{x\in B_\rho : \varphi^1(x)=0,\varphi^2(x)=0\}=:\Omega_\rho\,,
\]
and, again using \eqref{eq:decomp-k}, $\nabla\varphi^1(0)=\nabla{P}^1_1(0), \nabla\varphi^2(0)=\nabla{P}^2_1(0)$ are linearly independent.
Then by the implicit function theorem, noting that~$\psi\in C^{1,\alpha}$, there exists $\rho>0$ such that $\Omega_\rho$ is a~$C^{1,\alpha}$-submanifold of dimension $(n-2)$, concluding the proof.
\end{proof}


\subsection{Dimension estimates for \texorpdfstring{$\cZ_{\geq2}$}{Z>=2}}
For simplicity, let us assume the spinor $\varphi$ is defined on the unit ball $B_1$, as the argument is local.

Observe that, in terms of the components~$\varphi=(\varphi^1,\ldots,\varphi^{2N}) $, there holds
\begin{equation}\label{eq:inters}
\cZ_{\geq 2}(\varphi)=\bigcap^{2N}_{j=1}\cZ_{\geq 2}(\varphi^j)\,.
\end{equation}
That is, the singular set of the spinor is the intersection of the singular sets of its components. 
Moreover, in terms of~\eqref{eq:decomp-k} we have
\begin{equation}\label{eq:nodaltwo}
\begin{split}
\cZ_{\geq2}(\varphi^j)&=\{x_0\in B_1\,:\, \varphi^j(x_0)=0\,,\, \varphi^j(x)=P^j_k(x-x_0)+Q^j_k(x-x_0),\, \mbox{for some $k\geq2$}\} \\
&=\{x_0\in B_1\,:\, \varphi^j(x_0)=0\,,\,\nabla\varphi^j(x_0)=0\} \,.
\end{split}
\end{equation}
Thus we are led to prove the desired dimension estimate for each set $\cZ_{\geq 2}(\varphi^j)$.
This is achieved using an argument from \cite{CaffarelliFriedman1985partial}, as illustrated below.  
\smallskip

For fixed $k\geq 2$, denote by $\cZ_k(\varphi)$ the set of points for which the homogeneous polonomial part in~\eqref{eq:decomp-k} has degree~$k$.  
This essentially means that~$\varphi$ vanishes there to the order~$k$, a concept which will be made clear later via the frequency function. 
It is evident that
\begin{equation}\label{eq:Zorder}
\cZ_{\geq 2}=\bigcup_{k\geq 2}\cZ_k\,.
\end{equation}
Theorem \ref{thm:DimEstimate} is a consequence of the following result, thanks to \eqref{eq:inters} and \eqref{eq:Zorder}.

\begin{thm}\label{thm:Hmeas}
There holds
\begin{equation}\label{eq:Hmeas}
\cH^{n-2+\gamma}(\cZ_k(\varphi^j))=0\,,\qquad \forall \gamma>0\,,\, \forall j\in\{1,\ldots, 2N\}\,,
\end{equation}
where $\cH^\beta$ denotes the $\beta$-dimensional Hausdorff measure.
\end{thm}
Observe that \cite[Thm. 2.1]{CaffarelliFriedman1985partial} applies to $\cZ_k(\varphi^j)$. 
It essentially proves that the set is locally contained in a cusp-like neighborhood of a $(n-2)$-plane. In our context, such result reads 
\begin{lemma}\label{lem:cusplike}
Take $0<r<1$. If $x_0\in \cZ_k(\varphi^j)\cap B_{1-r}$, there exists a ball $B_\delta(x_0)$ and an $(n-2)$-plane $\pi_{x_0}$ passing through $x_0$ such that
\[
\cZ_k\cap B_{\delta}(x_0)\cap B_{1-r/2}\subseteq K(x_0)\,,
\]
where $K(x_0)=\{x\,:\, C\vert x-x_0\vert^{1+\eps}> d(x,\pi_{x_0})\}$, for some $0<\eps<1$. Here $d(\cdot,\cdot)$ denotes the distance of a point from a set, and the constants $C,\delta>0$ only depend on $r$ and on a lower bound on $\Vert P_k\Vert_{L^2(B_1)}$.
\end{lemma}
The proof of the above result essentially relies on the decomposition \eqref{eq:decomp-k} and on the properties of homogeneous harmonic polynomials, which represent the leading order term in that expansion. Since the nodal set is then locally `squeezed' on a $(n-2)$-plane, we can use this information to estimate its Hausdorff dimension by a simple covering argument. Before doing so, we need a technical result.
\begin{lemma}\label{lem:count}
Let $K\subseteq \R^n$ be the unit cube and $\pi$ an $(n-2)$-plane. Then for any small $\delta>0$ there exists $N(\delta)$ cubes of sides $2\delta$ covering $\pi\cap K$, with
\begin{equation}\label{eq:cubeN}
N(\delta)\leq \left(\frac{\sqrt{n}+1}{2\delta}\right)^{n-2}\,.
\end{equation} 
\end{lemma}
\begin{proof}
Consider, for simplicity, coordinates such that $\pi=\{x_{n-1}=x_n=0\}$. Notice that $\pi\cap K$ is contained in an $(n-2)$-cube of side $\sqrt n$. Covering $\pi\cap K$ with cubes, with disjoint interior, of sides $2\delta$ and parallel to the coordinate axis, for the total volume of the covering one gets
\[
(2\delta)^{n-2}N(\delta)\leq(\sqrt n + 1)^{n-2}\,,
\]
$N(\delta)$ being the number of the cubes involved, and \eqref{eq:cubeN} follows.
\end{proof}

\begin{proof}[Proof of Theorem \ref{thm:Hmeas}]
We start by considering, for fixed $0<\mu<1$, $\eta>0$, the set
\[
\cZ^{\mu,\eta}_{k}(\varphi^j)=\{x\in \cZ_k(\varphi^j)\cap B_{1-\mu}\,:\, \Vert P^j_k\Vert_{L^2(B_1)}\geq\eta\} \,.
\]
as, in order to apply the previous Lemma \ref{lem:cusplike}, we need a positive lower bound on $\Vert P^j_k\Vert_{L^2(B_1)}$. By such result, we know that there exist $\eps,C_0,l_0>0$ such that 
\[
S:=\cZ^{\mu,\eta}_{k}(\varphi^j)
\]
 is contained in the $(C_0l)^{1+\eps}$-neighborhood of an $(n-2)$-plane $\pi$, for any $l\leq l_0$. 
 
 We now argue by induction, constructing a sequence of coverings of $S$ with cubes of shrinking sides. For a suitable sequence $(m_k)$ (to be chosen later), cover the set $S$ with cubes of side $2^{-m_k}$, at step $k$. Passing to step $k+1$, scale the sides of the cubes by $\delta=2^{-m_k\eps}$ and apply Lemma \ref{lem:count}.
 Then by \eqref{eq:cubeN} the number $N_k$ of cubes in the covering at step $k$ satisfies
\[
N_{k+1}\leq N_{k}\left(\frac{\sqrt n +1}{2^{-m_k\eps+1}} \right)^{n-2}\,.
\]
Choosing $m_{k+1}=m_{k}(1+\eps)$, we thus get
\[
N_{k+1}\leq N_{k}\left(\frac{\sqrt n +1}{2^{-(1+\eps)^k\eps+1}} \right)^{n-2}\,.
\]
Iterating the estimate, one finds
\[
N_{k+1}\leq \left(\frac{\sqrt n+1}{2} \right)^{n-2}\left(2^{\eps\sum^k_{j=1}(1+\eps)^j} \right)^{n-2}\leq C^k 2^{(1+\eps)^{k+1}(n-2)}\,,
\]
$C=C(n)>0$ being a dimensional constant. Using the above inequality, we can estimate the $(n-2+\gamma)$-measure of $S$ as
\[
\begin{split}
\cH^{n-2+\gamma}(S)&\leq\liminf_{k\to\infty}(2^{-(1+\eps)^{k+1}})^{n-2+\gamma}N_{k+1} \\
&=\lim_{k\to\infty} C^k 2^{-(1+\eps)^{k+1}\gamma}=0\,,\qquad\forall \gamma>0\,.
\end{split}
\]
Taking $\eta\to0^+$ one thus gets
\[
\cH^{n-2+\gamma}(\cZ_k\cap B_{1-\mu})=0\,,\qquad\forall\gamma>0\,,
\]
and \eqref{eq:Hmeas} finally follows letting $\mu\to0^+$.
\end{proof}

\section{A frequency function for spinors}\label{sec:freq}

In this section we study a version of Almgren's frequency function for spinors \eqref{eq:freq}, as anticipated in Section \ref{sec:intro}. 
Our aim is to make the notion of vanishing order at point clear for a non-smooth spinor and to obtain a quantitative control on it. In particular, we show that such a function is globally bounded on $M$ (or on a compact region, if the manifold itself is not compact). 
This strategy is well-known for second order elliptic operators on scalar functions, see \cite{CheegerNaberValtorta2015critical, NaberValtorta2017volume} and references therein, as well as~\cite{Taubes2014zero,HaydysWalpuski2015compactness} for the spinorial case. 
However, here, inspired by \cite{GarofaloLin1986monotonicity}, we choose to use a version of frequency function which is slightly different from the one typically used in the mentioned references on Dirac equations. 
We believe that this choice is more suited for our purposes and allows to deal with a larger classe of equations.

We point out that in order to incorporate the influence of the local geometry, we choose to work in the original Riemannian metric~$g$. Alternatively, of course, one could equivalently work in Euclidean frames, by carefully keeping track of the Christoffel symbols.

Generally speaking, the frequency function at a point $x\in M$ usually takes the form 
\begin{align}\label{eq:frequency}
 N(x,r)= \frac{r D(x,r)}{ H(x,r)}\,,\quad r>0\,,
\end{align}
with~$D(x,r)$ and~$H(x,r)$ to be suitably chosen according to the problem under consideration. 
For the moment the basepoint~$x$ will be fixed, thus we will omit it and simply write~$H(r), D(r)$ and~$N(r)$, regarding them as functions in~$r$. Later on we will also consider the dependence on the basepoint.

If~$H$ and~$D$ are nonzero (even only for small~$r>0$) and differentiable, then 
\begin{align}
 \frac{N'}{N}=\frac{1}{r}+\frac{D'}{D}-\frac{H'}{H}.
\end{align}
This can be used to show certain (almost) monotonicity of~$N(r)$ by analyzing the above differential identity. We refer the reader, for instance, to \cite{Almgren1979Dirichlet, GarofaloLin1986monotonicity, Han1994singular} for the case of second order elliptic equations. 
\medskip

In the argument below  local coordinates are understood to be the normal ones.

Moreover, in this section we will assume that the map $V$ in \eqref{eq:Dirac eq-V} satisfies \eqref{eq:Vf} or \eqref{eq:Vq}.

\subsection{The denominator function}
The choice of the denominator~$H$ is quite standard. That is, we take
\begin{align}\label{eq:H}
 H(r)=\int_{\p B_r} |\psi|^2\ds_r 
 =\int_{\sph^{n-1}} |\psi(r,\theta)|^2 r^{n-1}\sqrt{\det(w_{ij}(r,\theta))} \dd\theta.   
\end{align}
Its derivative is 
\begin{align}\label{eq:H'}
 H'(r)
 =&\int_{\sph^{n-1}} (n-1) r^{n-2}|\psi(r,\theta)|^2 \sqrt{\det(w_{ij}(r,\theta))} 
 + 2\Abracket{\snabla_{\nu}\psi,\psi}r^{n-1}\sqrt{\det(w_{ij}(r,\theta))} \dd\theta \\
 &\qquad +\int_{\sph^{n-1}} |\psi(r,\theta)|^2 r^{n-1} \frac{\p}{\p r}\sqrt{\det(w_{ij}(r,\theta))} \dd\theta \\
 =&\frac{n-1}{r}H(r) + 2\int_{\p B_r}\Abracket{\snabla_{\nu}\psi,\psi}\ds_r 
 + \int_{\p B_r} |\psi|^2 W(r,\theta) \ds_r .
\end{align}

\subsection{The numerator function}

Dealing with an equation of general form as \eqref{eq:Dirac eq-V}, the choice of the suitable numerator function is not obvious. 
This is actually the main difficulty to generalized the frequency method to PDE systems. 
For Dirac equations, some generalizations are considered in the literature, see e.g.~\cite{Taubes2014zero, HaydysWalpuski2015compactness, DoanWalpuski2021existence}, in the setting of Seiberg-Witten equations.
However, since we do not find them convenient for our purposes, we prefer to follow \cite{GarofaloLin1986monotonicity} and by analogy we consider the function
\begin{align}\label{eq:D}
 D(r)= \int_{\p B_r} \Abracket{\snabla_{\nu}\psi,\psi}\ds_r. 
\end{align}
With this choice \eqref{eq:H'} reads as 
\begin{align}
 H'(r)=\frac{n-1}{r}H(r)+2D(r)+ \int_{\p B_r} |\psi|^2 W(r,\theta)\ds_r. 
\end{align}
Moreover, thanks to Gauss--Green formula and Lichnerowicz formula, we have 
\begin{align}
 D(r)= \int_{\p B_r}\Abracket{\snabla_\nu\psi,\psi}\ds_r
 =&\int_{B_r} |\snabla\psi|^2 -\Abracket{\snabla^*\snabla\psi,\psi}\dv_g \\
 =&\int_{B_r} |\snabla\psi|^2+\frac{\Scal}{4}|\psi|^2 - \Abracket{\D V(\psi),\psi}\dv_g. 
 \end{align}
Thus by the co-area formula
\begin{align}\label{eq:D'}
 D'(r)=\int_{\p B_r} |\snabla\psi|^2 -\Abracket{\snabla^*\snabla\psi,\psi}\ds_r.
\end{align}
We remark again that~$\snabla^*\snabla\psi$ is defined pointwisely for~$V$ among the particular types under consideration. 

Moreover, by Lichnerowicz formula, 
\begin{align}
 D(r)=\int_{B_r} |\snabla\psi|^2+\frac{\Scal}{4}|\psi|^2 -|\D\psi|^2\dv_g 
      -\int_{\p B_r} \Abracket{\gamma(\nu)\D\psi,\psi}\ds_r.
\end{align}
In general, the above boundary integral above is troublesome and we do not know how to deal with it in full generality.. 
However, for~$V$ satisfying \eqref{eq:Vf} or \eqref{eq:Vq}, we can still handle it. 
Indeed, if~$V$ is of type \eqref{eq:Vf}, then due to the skew-symmetry of Clifford multiplications by tangent vectors, 
\begin{align}
 \Abracket{\gamma(\nu)\D\psi,\psi}
 =f(\psi)\Abracket{\gamma(\mu)\psi, \psi}=0,
\end{align}
hence the boundary integral drops out in this case. 
Then the co-area formula tells 
\begin{align}
 D'(r)=\int_{\p B_r} |\snabla\psi|^2 +\frac{\Scal}{4}|\psi|^2 - |V(\psi)|^2\ds_r, 
\end{align}
and in particular, denoting~$\|\Scal\|=\|\Scal\|_{L^\infty}$, we have 
\begin{align}
 D'(r) -\int_{B_r} |\snabla\psi|^2 \ds_r
 =&\int_{\p B_r} \frac{\Scal}{4}|\psi|^2 - |V(\psi)|^2 \ds_r  \\
 \ge&  -\parenthesis{\frac{\|\Scal\|}{4} +\CLip } \int_{\p B_r} |\psi|^2 \ds_r 
    = -\parenthesis{\frac{\|\Scal\|}{4} +\CLip } H(r). 
\end{align}
If~$V$ is of type (Vq), then 
\begin{align}
 D'(r)=\int_{\p B_r} |\snabla\psi|^2 +\frac{\Scal}{4}|\psi|^2 - \Abracket{\D V(\psi),\psi}\ds_r. 
\end{align}
and by the condition~\eqref{eq:Vq} we have 
\begin{align}
 D'(r)-\int_{\p B_r} |\snabla\psi|^2\ds_r 
 \ge -\parenthesis{\frac{\|\Scal\|}{4} + \Cq(1+\|\snabla\psi\|_{L^\infty(K)})} H(r). 
\end{align}
Thus in the cases under consideration we have 
\begin{align}\label{eq:boundary-Laplace}
  D'(r)-\int_{\p B_r} |\snabla\psi|^2\ds_r 
 \ge -\CbL  H(r)
\end{align}
for some constant
\begin{align}
 \CbL=\CbL(\|\Scal\|, \CLip, \Cq, \|\snabla\psi\|_{L^\infty})>0. 
\end{align}
This fact will be of crucial importance in later arguments.    
\medskip

In contrast to the case of the scalar functions treated in~\cite{GarofaloLin1986monotonicity}, here it is unclear whether the numerator function~$D(r)$ is positive, even for small~$r>0$. 
Actually, what we really need is the local almost positivity of \eqref{eq:freq}, in the sense that for any~$C_N>0$, the quantity~$N(r)+C_N$ is positive for small~$r>0$.
\begin{prop}
Fix $C_N>0$. There exists $\rPos>0$ such that for any $0<r<\rPos$ there holds
\begin{equation}\label{eq:Nalmostpositive}
N(r)+C_N(r)>0\,.
\end{equation}
\end{prop}
\begin{proof}
 Observe that \eqref{eq:Nalmostpositive} is equivalent to the local positivity of~$r D(r)+ C_N H(r)$ for small~$r>0$.  For~$V$ of type \eqref{eq:Vf} we find
\[
\begin{split}
 rD(r) & + C_N H(r)\\
 =& \, r\int_{B_r} |\snabla\psi|^2 +\frac{\Scal}{4}|\psi|^2 - |V(\psi)|^2 \dv_g  + C_N\int_{\p B_r} |\psi|^2 \ds_r\\
 \ge& \, r\int_{B_r}|\snabla\psi|^2 \dv_g -\parenthesis{\frac{\|\Scal\|}{4}+\CLip} r\int_{B_r} |\psi|^2 \dv_g 
      +C_N \int_{\p B_r} |\psi|^2 \ds_r \\
 \ge& \, r\int_{B_r} |\snabla\psi|^2\dv_g -\parenthesis{\frac{\|\Scal\|}{4} + \CLip }\CHardy^2 r^3 \int_{B_r} |\snabla\psi|^2 \dv_g  \\
  & -\parenthesis{\frac{\|\Scal\|}{4}+\CLip} \CHardy r^2 \int_{\p B_r} |\psi|^2\ds_r
  +C_N \int_{\p B_r} |\psi|^2 \ds_r \\
 \ge&\,  r\parenthesis{1-\parenthesis{\frac{\|\Scal\|}{4} +\CLip}\CHardy^2 r^2 } \int_{B_r} |\snabla\psi|^2\dv_g \\
     & +\parenthesis{C_N- \parenthesis{\frac{\|\Scal\|}{4} +\CLip} \CHardy r^2 } \int_{\p B_r} |\psi|^2 \ds_r. 
\end{split}
\]
Thus there exists~$\rPos>0$ and~$\CPos>0$ such that for~$0<r<\rPos$ 
\begin{equation}\label{eq:almost positivity}
 r D(r) + C_N H(r)
 \ge \CPos r \int_{B_r} |\snabla\psi|^2 \dv_g 
     + \CPos \int_{\p B_r} |\psi|^2 \ds_r 
\end{equation}
which is positive. 

Meanwhile, if~$V$ is of type \eqref{eq:Vq}, since~$\psi\in C^{1,\alpha}$, the gradient~$|\snabla\psi|$ is locally uniformly bounded, so that 
\begin{align}
 |\D V(\psi)| \le \Cq \parenthesis{ |\psi|+ |\psi||\snabla\psi|} \le \Cq' |\psi| 
\end{align}
near the zero section locally; for simplicity we will write~$\Cq$ for~$\Cq'$, hence the assumption~\eqref{eq:Vq} becomes (locally) 
\begin{align}
 |\D V(\psi)| \le \Cq |\psi|. 
\end{align}
In this case, 
\begin{align}
 rD(r) &+ C_N H(r)\\
 =& r\int_{B_r} |\snabla\psi|^2 +\frac{\Scal}{4}|\psi|^2 - \Abracket{\D V(\psi),\psi} \dv_g  + C_N\int_{\p B_r} |\psi|^2 \ds_r\\
 \ge&r\int_{B_r} |\snabla\psi|^2 \dv_g -\parenthesis{\frac{\|\Scal\|}{4}+ \Cq}\int_{B_r}|\psi|^2\dv_g 
    +C_N\int_{\p B_r} |\psi|^2\ds_r
\end{align}
and then the same argument as above works also in this case. 
\end{proof}

We remark that the constant~$C_N$ in~\eqref{eq:almost positivity} can be very small (and then choose~$\rPos$ accordingly), reflecting the fact that~$D(r)$ is ``nonnegative'' locally near~$r=0$.

\subsection{Almost monotonicity of \texorpdfstring{$N(r)+C_N$}{N(r)+CN}}

Now fix a positive constant~$C_N$ and consider the function~$N(r)+C_N$. 
Then 
\begin{align}
 \frac{\dd}{\dd r}(N(r)+C_N)
 =N'(r)
 =& \parenthesis{\frac{1}{r}+\frac{D'}{D}-\frac{H'}{H}} N(r)
\end{align}
as long as the RHS is defined. 
\begin{lemma}\label{lemma:almost monotonicity-differential form}
 There exist a small radius~$\rAM>0$ and constants~$\CAM>0$ and~$\beta\in (0,1) $ such that 
 \begin{align}\label{eq:almost monotonicity for frequency}
  N'(r)\ge -\CAM r^\beta \parenthesis{N(r)+C_N}, \qquad \forall r\in (0, \rAM). 
 \end{align}

\end{lemma}

\begin{proof}
By~\eqref{eq:H'}, the inequality~\eqref{eq:almost monotonicity for frequency} is equivalent to
 \begin{align}
  (2-n+ \CAM r^{1+\beta})DH+ \CAM C_N r^{\beta}H^2  + rHD'
  \ge 2r D^2 + rD \int_{\p B_r}|\psi|^2 W(r,\theta)\ds_r. 
 \end{align}
 By Cauchy--Schwartz we have 
 \begin{align}
  D(r)^2= \parenthesis{\int_{\p B_r} \Abracket{\snabla_\nu\psi, \psi}\ds_r }^2
  \le \parenthesis{\int_{\p B_r}|\snabla_\nu\psi|^2\ds_r}\parenthesis{\int_{\p B_r} |\psi|^2\ds_r}
 \end{align}
 so that 
 \begin{align}
  2rD(r)^2 \le 2 r H(r)\parenthesis{\int_{\p B_r}|\snabla_\nu \psi|^2\ds_r}. 
 \end{align}

 To handle with~$D'(r)$ we use the Pohozaev type formula for spinors~\eqref{eq:Pohozaev-spinor} to obtain 
\begin{align}
 \int_{\p B_r} |\snabla\psi|^2 \ds_r 
 =& \frac{n-2}{r}\int_{\p B_r} \Abracket{\snabla_\nu\psi,\psi}\ds_r 
  +2\int_{\p B_r}|\snabla_\nu\psi|^2 \ds_r \\
  & +\frac{2}{r}\int_{B_r} \Abracket{R_{\vec{x},e_j}\psi, \snabla_j\psi}\dv_g 
   +\frac{1}{r}\int_{B_r} \Abracket{2\snabla_{\vec{x}}\psi+(n-2)\psi, \snabla^*\snabla\psi}\dv_g.
\end{align}
Hence~\eqref{eq:D'} becomes 
\begin{align}
 D'(r)
 =&\frac{n-2}{r}D(r)+ 2\int_{\p B_r}|\snabla_\nu\psi|^2\ds_r 
   -\int_{\p B_r}\Abracket{\snabla^*\snabla\psi,\psi}\ds_r \\
  & +\frac{2}{r}\int_{B_r} \Abracket{R_{\vec{x},e_j}\psi, \snabla_j\psi}\dv_g 
   +\frac{1}{r}\int_{B_r} \Abracket{2\snabla_{\vec{x}}\psi+(n-2)\psi, \snabla^*\snabla\psi}\dv_g. 
\end{align}
Thus it suffices to show 
\begin{multline}\label{eq:almost monotonicity-i}
 \CAM r^{1+\beta} HD + \CAM C_N r^\beta H^2 
 -r H(r)\parenthesis{\int_{\p B_r} \Abracket{\psi,\snabla^*\snabla\psi}\ds_r}
 \\
 +H(r)\parenthesis{ \int_{B_r} 2\Abracket{R_{\vec{x},e_j}\psi,\snabla_j \psi}  + \Abracket{2\snabla_{\vec{x}}\psi + (n-2)\psi, \snabla^*\snabla\psi} \dv }\\ 
 \ge r D(r)\int_{\p B_r} |\psi|^2 W(r,\theta)\ds_r.
\end{multline}

At place where~$H(r)=0$,~\eqref{eq:almost monotonicity-i} trivially holds, so w.l.o.g. we may assume~$H(r)>0$ and~\eqref{eq:almost monotonicity-i} reduces to 
\begin{multline}\label{eq:almost monotonicity-ii}
 \CAM r^{1+\beta} D + \CAM C_N r^\beta H 
 -r\parenthesis{\int_{\p B_r} \Abracket{\psi,\snabla^*\snabla\psi}\ds_r}
 \\
 +\parenthesis{ \int_{B_r} 2\Abracket{R_{\vec{x},e_j}\psi,\snabla_j \psi}  + \Abracket{2\snabla_{\vec{x}}\psi + (n-2)\psi, \snabla^*\snabla\psi} \dv }\\
 \ge rD(r) \frac{\int_{\p B_r} |\psi|^2 B(r,\theta)\ds_r}{\int_{\p B_r}|\psi|^2\ds_r}.
\end{multline}
 
Recall that from~\eqref{eq:D'} and~\eqref{eq:boundary-Laplace} we have 
\begin{align}
 -r\parenthesis{\int_{\p B_r} \Abracket{\psi,\snabla^*\snabla\psi}\ds_r}
 \ge -\CbL r H(r). 
\end{align}

\

Next we estimate the integrals over~$B_r$. 
By the Hardy type inequality~\eqref{eq:Hardy-spinor}, we get
\begin{align}
 \int_{B_r}2\Abracket{R_{\vec{x},e_j}\psi,\snabla_j\psi}\dv
 \le & 2|\vec{x}|\|R\| \int_{B_r} |\psi||\snabla\psi|\dv 
 \le r\|R\|\int_{B_r}\frac{1}{r}|\psi|^2 + r|\snabla\psi|^2\dv  \\
 \le& r\|R\| \parenthesis{ \CHardy  H(r)+(1+ \CHardy^2 ) r\int_{B_r}|\snabla\psi|^2\dx} \\
 \le& \CHardy \|R\| r H(r) + \|R\|(1+ \CHardy^2) r^2 \int_{B_r}|\snabla\psi|^2\dx. 
\end{align}
By Lichnerowicz formular, we have 
\begin{align}
 \int_{B_r}  & 2\Abracket{\snabla_{\vec{x}}\psi, \snabla^*\snabla\psi}  \dv_g
 = \int_{B_r} 2\Abracket{\snabla_{\vec{x}}\psi, \D V(\psi)-\frac{\Scal}{4}\psi}\dv_g. 
\end{align}
For~$V$ of type \eqref{eq:Vf}, recalling that~\eqref{eq:f1}, we have  
\begin{align}
 \int_{B_r}  & 2\Abracket{\snabla_{\vec{x}}\psi, \snabla^*\snabla\psi}  \dv_g
 \ge -2r\int_{B_r}  \Cf |\psi|^\tau |\snabla\psi|^2 +  \parenthesis{\frac{\|\Scal\|}{4}+\Cf}|\psi||\snabla\psi|\dv_g\\
 \ge& -2\Cf r\int_{B_r} |\psi|^\tau |\snabla\psi|^2 \dv_g 
       -\parenthesis{\frac{\|\Scal\|}{4} +\Cf} r \int_{B_r} \frac{|\psi|^2}{r} +  r|\snabla\psi|^2 \dv_g \\
 \ge& -2\Cf r\int_{B_r} |\psi|^\tau |\snabla\psi|^2 \dv_g \\
    &  -\parenthesis{\frac{\|\Scal\|}{4}+\Cf} r  \parenthesis{ \CHardy H(r)+(1+\CHardy^2) r \int_{B_r} |\snabla\psi|^2\dv_g } \\
 \ge& -\parenthesis{\frac{\|\Scal\|}{4}+\Cf} \CHardy r H(r) \\
    &- r\parenthesis{ 2\Cf \|\psi\|^\tau_{L^\infty(B_r)} + \parenthesis{\frac{\|\Scal\|}{4} +\Cf }(1+\CHardy^2)r }\int_{B_r}|\snabla\psi|^2\dv_g.
\end{align}
Note that~$\|\psi\|^\tau_{L^\infty(B_r)}\le C r^{(1+\alpha)\tau}$
For~$V$ of type \eqref{eq:Vq}, the situation is similar but easier, since~\eqref{eq:Vq} is stronger than~\eqref{eq:f1}. 

We turn to the estimate for 
\begin{align}
 (n-2)\int_{B_r}  \Abracket{\psi,\snabla^*\snabla\psi}\dv_g
 =& (n-2)\int_{B_r} \Abracket{\psi, \D^2\psi+\frac{\Scal}{4}\psi }\dv_g.
\end{align}
For~$V$ of type \eqref{eq:Vf} integrating by parts one obtains 
\begin{align}
 (n-2)\int_{B_r}  &\Abracket{\psi,\snabla^*\snabla\psi}\dv_g
 =(n-2)\int_{B_r} |\D\psi|^2 + \frac{\Scal}{4}|\psi|^2 \dv_g \\
 \ge &- (n-2)\parenthesis{ \CLip+\frac{\|\Scal\|}{4} } \int_{B_r} |\psi|^2 \dv_g\\
 \ge &- (n-2)\parenthesis{ \CLip+\frac{\|\Scal\|}{4} } \parenthesis{ \CHardy r H(r) + \CHardy^2 r^2 \int_{B_r} |\snabla\psi|^2 \dv_g};
\end{align}
while for~$V$ of type \eqref{eq:Vq} we use~\eqref{eq:Vq} to get 
\begin{align}
 (n-2)\int_{B_r} & \Abracket{\psi,\snabla^*\snabla\psi}\dv_g
 \ge- (n-2)\parenthesis{\Cq+\frac{\|\Scal\|}{4}}\int_{B_r} |\psi|^2\dv_g \\
 \ge&- (n-2)\parenthesis{ \Cq+\frac{\|\Scal\|}{4} } \parenthesis{ \CHardy r H(r) + \CHardy^2 r^2 \int_{B_r} |\snabla\psi|^2 \dv_g}. 
\end{align}

\

To summarize, we see that
\begin{align}
  \CAM r^{1+\beta} D +& \CAM C_N r^\beta H 
 -r\parenthesis{\int_{\p B_r} \Abracket{\psi,\snabla^*\snabla\psi}\ds_r}
 \\
 +&\parenthesis{ \int_{B_r} 2\Abracket{R_{\vec{x},e_j}\psi,\snabla_j \psi}  + \Abracket{2\snabla_{\vec{x}}\psi + (n-2)\psi, \snabla^*\snabla\psi} \dv } \\
 \ge& \CAM  r^\beta\parenthesis{ r D(r)+ C_N H(r)} - C r H(r) \\
    &- C(r+ \|\psi\|^\tau_{L^\infty(B_r)}) r \int_{B_r } |\snabla\psi|^2 \dv_g \\
 \ge& \CAM r^\beta \parenthesis{ \CPos r \int_{B_r} |\snabla\psi|^2 \dv_g + \CPos H(r) } -Cr H(r) \\
    &- C(r+ \|\psi\|^\tau_{L^\infty(B_r)}) r \int_{B_r } |\snabla\psi|^2 \dv_g.
\end{align}
where in the last inequality we used~\eqref{eq:almost positivity}, and that $C$ is a universal constant. 
Moreover, note that the right hand side of~\eqref{eq:almost monotonicity-ii} can be estimated by 
\begin{align}
 r D(r)\frac{\int_{\p B_r}|\psi|^2 W(r,\theta)\ds_r}{\int_{\p B_r}|\psi|^2\ds_r} 
 \le \CCoord r^2 D(r)
\end{align}
and~$D(r)$ can be in turn estimated, for~$V$ of type (Vf) or (Vq), by 
\begin{align}
 D(r) \le C \int_{B_r} |\snabla\psi|^2\dv_g +  C H(r). 
\end{align}
Recall that~$\psi\in C^{1,\alpha}$ and then $\|\psi\|^\tau_{L^\infty(B_r)} \le C r^{(1+\alpha)\tau}$. 
Thus we can choose~$\beta\in (0,1)$ and~$\rAM>0$ small,~$\CAM$ large enough, so that~\eqref{eq:almost monotonicity for frequency} holds. 
 
\end{proof}

\subsection{A unique continuation principle}
As a first consequence of~\eqref{eq:almost monotonicity for frequency}, we first show that~$\psi$ cannot vanish in an open subset of~$\Omega$. 

\

First recall the formula~\eqref{eq:H'}:
\begin{align}
 H'(r)=\frac{n-1}{r} H + 2D(r) + \int_{\p B_r} |\psi|^2 W(r,\theta)\ds_r
 \ge \frac{2}{r}(r D(r)+ \frac{n-1}{2}H(r) - \frac{\CCoord r^2}{2}H(r))
\end{align}
which is nonnegative for small~$r\in (0, \rPos)$; and $H(r)$ is thus locally non-decreasing. 

We argue by contradiction and assume that~$\psi$ vanishes in~$B_{\rho_1}(y)$ but does not vanish identically in~$B_{\rho_2}(y)$ for some~$y\in \Omega$ and~$0<\rho_1 < \rho_2 < \rPos(y)$ (note that~$\rPos$ may taken to be independent of~$y$). 
Then, by the above monotonicity of~$H(r)$, there exists~$\rho_*\in [\rho_1, \rho_2]$ such that 
\begin{itemize}
 \item $H(\rho)=0$ for all~$\rho\le \rho_*$;
 \item $H(\rho)>0$ for~$\rho\in (\rho_*, \rPos)$.
\end{itemize}
Then for any~$\rho_* < s< t< \rPos$,  
\begin{align}
 \log H(t)-\log H(s)
 =& \int_s^t (\log H)'(r)\dd r 
   =\int_s^t \frac{n-1}{r} + \frac{2}{r} N(r)+ \frac{\int_{\p B_r}|\psi|^2 W(r,\theta)\ds_r}{\int_{\p B_r}|\psi|^2 \ds_r}\dd r \\
 \le& (n-1+ 2\sup_{r\in [s,t]} N(r)) \int_s^t \frac{1}{r}\dd r
  +\int_s^t \CCoord r\dd r\\
 \le& (n-1+2\sup_{r\in[s,t]} N(r)) \log\frac{t}{s} 
 +\frac{1}{2}\CCoord (t^2- s^2),
\end{align}
namely, we obtain a Harnack type inequality:
\begin{align}
 e^{-\frac{\CCoord}{2}t^2}H(t)\le e^{-\frac{\CCoord}{2}s^2}\parenthesis{\frac{t}{s}}^{n-1+2\sup_{[s,t]}N} H(s). 
\end{align}
Meanwhile for~$\rho_* < r< \rPos$,~$N(r)$ is well-defined and positive, and 
\begin{align}
 \frac{\dd}{\dd r} \log(N(r)+C_N)
 \ge -\CAM  r^\beta
\end{align}
from which it follows that 
\begin{align}
 \log\frac{N(t)+C_N}{N(s)+C_N} \ge -\frac{\CAM }{\beta+1} (t^{\beta+1}-s^{\beta+1})
\end{align}
and hence 
\begin{align}
 N(s)+C_N \le \exp\parenthesis{\frac{\CAM }{\beta+1}(t^{\beta+1}-s^{\beta+1})} (N(t)+C_N). 
\end{align}
In particular,~$N(\rho)$ (for~$\rho\in (\rho_*, r_2)$) is uniformly bounded by 
\begin{align}
 \sup_{\rho\in [s,t]} N(r) \le \exp\parenthesis{\frac{\CAM }{\beta+1}(r_2^{\beta+1}-\rho_*^{\beta+1})} (N(r_2)+C_N) -C_N\equiv N^*. 
\end{align}
and consequently 
\begin{align}
 e^{-\frac{\CCoord}{2}t^2}H(t)\le e^{-\frac{\CCoord}{2}s^2}\parenthesis{\frac{t}{s}}^{n-1+ 2N^*} H(s).
\end{align}
Letting~$s\searrow \rho_*$ we would get 
\begin{align}
 H(t)=0
\end{align}
for any~$t\in (\rho_*, \rPos)$, which is a contradiction. 
Thus the unique continuation property holds for~$\psi$ and~$H(r)>0$ for any~$r\in (0,\rPos)$.

\subsection{Almost monotonicity for~\texorpdfstring{$N(r)$}{N(r)}}

Now for a nonzero solution~$\psi$ of~\eqref{eq:Dirac eq-V}, assuming \eqref{eq:Vf} or \eqref{eq:Vq}, the frequency function~$N(y,r)$ is well-defined for~$y\in \mathcal{Z}(\psi)$ and~$0<r<\rPos$. 
Moreover, from Lemma~\ref{lemma:almost monotonicity-differential form} it follows that, for any~$0<s<t<\rPos$, 
\begin{align}
  N(s)+C_N \le \exp\parenthesis{\frac{\CAM }{\beta+1}(t^{\beta+1}-s^{\beta+1})} (N(t)+C_N).
\end{align}
Equivalently, this can be rephrased as follows.
\begin{lemma}\label{lem:monotonicity}
The function
\begin{align}
 \exp\parenthesis{\frac{\CAM }{\beta+1} s^{\beta+1}}(N(s)+ C_N)\,,\qquad s\in(0,\rPos)
\end{align}
is monotonely increasing. 
\end{lemma}
\begin{cor}\label{cor:Nlimit}
\begin{align}\label{eq:limN}
 \lim_{s\to 0^+} N(y, s)\eqqcolon N(y,0)
\end{align}
exists for all~$y\in \mathcal{Z}(\psi)$. 
\end{cor}

The frequency function acutally describes the vanishing order of the spinor~$\psi\in C^{1,\alpha}$. 
Indeed, if~$P=P_k$ is a~$\D_{g_0}$-harmonic spinor with the nonzero components consisting of homogeneous degree~$k$ polynomials, then, by Euler's theorem for homogeneous functions, 
\begin{align}
 \lim_{r\to 0^+} N_P(0, r) 
 =\lim_{r\to 0^+} \frac{ r\int_{\p B_r}\Abracket{\snabla_{\nu} P, P }\ds_r}{\int_{\p B_r} |P|^2 \ds_r}=k. 
\end{align}
This justifies the following 
\begin{dfn}
 Let $\psi$ be a solution to \eqref{eq:Dirac eq-f} and take $y\in\mathcal Z(\psi)$. The \emph{vanishing order} of $\psi$ at $y$ is defined as
	\begin{equation}\label{eq:vord}
	\mathcal O (y)=N_\psi(y,0):=\lim_{r\to 0^+}N_\psi(y,r)\,
	\end{equation}
	where the limit exists thanks to \eqref{eq:limN}.
\end{dfn}

Let~$\psi$ be a solution of~\eqref{eq:Dirac eq-V} and let~$x_0 =0 \in \cZ(\psi) $.  
Thanks to the decomposition~\eqref{eq:decomp} of~$\varphi=\beta^{-1}(\psi)$ in Lemma \ref{lem:decomp-spinor} at the point $x_0$,
\begin{align}
 \mathcal O(x_0)=\lim_{r\to0^+} N_\psi(x_0,r) 
 = \lim_{r\to0^+} r\frac{\int_{\p B_r} \langle\snabla_\nu \psi,\psi\rangle\,\ds_r}{\int_{\partial B_r}\vert\psi\vert^2\,\ds_r}
 = \lim_{r\to0^+} r\frac{\int_{\partial B_r} \langle\snabla_\nu P,P\rangle\,\ds_r}{\int_{\partial B_r}\vert P\vert^2\,\ds_r}=k 
\end{align}
where in the~$\varphi$ or~$P$ formulation, the metric is Euclidean. 
Thus we see that the degree of the homogeneous polynomial~$P$ is precisely the vanishing order of~$\psi$ at the basepoint of the local coordinates.
The frequency function is a powerful tool in studying the nodal sets.
In particular, it helps to control the vanishing order of a spinor~$\psi$ in a large scale, as shown in next subsection.

\subsection{Local uniform upper bound for the frequency}

We have seen that 
\begin{align}
 N(x,r) = \frac{ r D(x,r)}{H(x,r)}
\end{align}
is almost monotone increasing in $r$ (in the sense of Lemma \ref{lem:monotonicity}), where~$x$ is fixed. 
In this section we derive an upper bound for the frequency function, locally uniform in~$x$. 
\medskip

We will assume that the injectivity radius of~$M$ is larger than~$\rAM$, so that we are always referring to geodesic balls in normal coordinates, without worrying about overlapping. 

\begin{lemma}\label{lem:Nunifbound}
  Let~$\psi$ be a solution of~\eqref{eq:Dirac eq-V}, assuming \eqref{eq:Vf} or \eqref{eq:Vq}, and~$x_0\in\mathcal{Z}(\psi)$. 
  Then there exists~$C_4>0$ such that for any~$x\in B_{\rAM /4}(x_0)$ such that~$B_{2r}(x)\subset B_{\rAM }(0)$, we have 
 \begin{align}
  N(x,r)\le C_4( N(x_0, \rAM )+1). 
 \end{align}

\end{lemma}
This is a well-known result for harmonic functions and, more generally, for solutions to scalar second order elliptic equations, for which we refer the reader to~\cite{CheegerNaberValtorta2015critical,NaberValtorta2017volume} and references therein.
\begin{proof}
 For simplicity of notation we take~$x_0=0$ and~$\rAM =1$. 
 Thus we need to show that for~$|x|<\frac{1}{4}$ and~$2r<1-|x|$, 
 \begin{align}
  N(x,r)\le C_4(N(0,1)+1). 
 \end{align}
 We may further assume that~$\rAM (x)=\rAM $ for all~$x\in B_{\rAM /2}(0)$, so that the almost monotonicity always holds for~$r<\rAM $ and~$x\in B_{\rAM /2}(0)$. 
 
 Recall that 
 \begin{align}
  \frac{\p H(x,r)}{\p r}
  = \frac{n-1}{r}H(x,r) + 2D(x,r)
  +\int_{\p B_r(x)} |\psi|^2 W(r,\theta)\ds_r. 
 \end{align}
 Thus 
 \begin{align}
  \frac{\frac{\p }{\p r}H(x,r)}{ H(x,r)}
  = \frac{n-1}{r} + \frac{2}{r}N(x,r) + \mathscr{W}(x,r), 
  \quad \mbox{with }
   \mathscr{W}(x,r) 
  \coloneqq \frac{\int_{\p B_r(x)} |\psi|^2 W(r,\theta)\ds_r}{\int_{\p B_r(x)} |\psi|^2 \ds_r}
 \end{align}
 and~$|\mathscr{W}(x,r)| \le \CCoord r$.
 It follows that, for~$s,t\in (0, \rAM -|x|)$, 
 \begin{align}
  \frac{H(x,t)}{H(x,s)}
  =\parenthesis{\frac{t}{s}}^{n-1-2C_N} \exp\parenthesis{\int_s^t \frac{2(N(x,r)+ C_N)}{r}\dd r} 
  \exp\parenthesis{\int_s^t \mathscr{W}(r)\dd r}. 
 \end{align}

\

Let~$R_3 <R_2 < R_1 < 1-|x|$, which are to be chosen later. 
The argument involves changing of centers and radii. 
We split it into three steps. 

\
 
\noindent\textbf{Step 1}.
Since~$H(x,r)$ is non-decreasing in~$r\in (0,\rAM )$, we have 
\begin{align}
 &\int_{B_{R_1}(x)} |\psi|^2 \dv_g 
 \ge  \int_{B_{R_1}(x)\setminus B_{R_2}(x)}|\psi|^2 \dv_g
     =\int_{R_2}^{R_1} H(x,r)\dd r \\ 
 =& \int_{R_2}^{R_1} \dd r\cdot H(x,R_2)\parenthesis{\frac{r}{R_2}}^{n-1-2C_N}\exp\parenthesis{\int_{R_2}^r\frac{2(N(x,t)+C_N)}{t}\dd t} \exp\parenthesis{\int_{R_2}^r \mathscr{W}(x,t)\dd t}\\
 =&\frac{H(x,R_2)}{R_2^{n-1-2C_N}} 
   \int_{R_2}^{R_1}\dd r\cdot  r^{n-1-2C_N}\exp\parenthesis{\int_{R_2}^r\frac{2(N(x,t)+C_N)}{t}\dd t} \exp\parenthesis{\int_{R_2}^r \mathscr{W}(x,t)\dd t}.
\end{align}
As observed before, for~$r\in [R_2, R_1]$, 
\begin{align}
 \int_{R_2}^r \mathscr{W}(x,t)\dd t
 \ge -\frac{\CCoord }{2}(r^2 - R_2^2)
 \ge -\frac{\CCoord }{2}(R_1^2 - R_2^2), 
\end{align}
and by the almost monotonicity of frequency function, 
\begin{align}
 \int_{R_2}^r\frac{2(N(x,t)+C_N)}{t}\dd t
 \ge& 2e^{\frac{C_{am}}{\beta+1}(R_2^{\beta+1} - R_1^{\beta+1})}(N(x, R_2)+C_N)\int_{R_2}^r \frac{1}{t}\dd t\\
 =& 2e^{\frac{C_{am}}{\beta+1}(R_2^{\beta+1} - R_1^{\beta+1})}(N(x, R_2)+C_N)\ln\frac{r}{R_2}.
\end{align}
Thus, by denoting~$A_1 \equiv 2e^{\frac{C_{am}}{\beta+1}(R_2^{\beta+1} - R_1^{\beta+1})}(N(x, R_2)+C_N)$, we have 
\begin{align}\label{eq:lower bound for R1}
 \int_{B_{R_1}(x)} |\psi|^2 \dv_g 
 \ge& \frac{H(x,R_2)}{R_2^{n-1-2C_N}}e^{-\frac{\CCoord }{2}(R_1^2 -R_2^2)}\int_{R_2}^{R_1} r^{n-1-2C_N}\parenthesis{\frac{r}{R_2}}^{A_1} \dd r \\
 =& \frac{H(x,R_2)}{R_2^{n-1-2C_N}}\frac{e^{-\frac{\CCoord }{2}(R_1^2 -R_2^2)}}{R_2^{A_1}}\frac{1}{n-2C_N+A_1} \parenthesis{R_1^{n-2C_N+A_1}- R_2^{n-2C_N+A_2}} \\
 =& e^{-\frac{\CCoord }{2}(R_1^2 -R_2^2)}\frac{R_2 H(x,R_2)}{n-2C_N+A_1}\parenthesis{\parenthesis{\frac{R_1}{R_2}}^{n-2C_N+A_1}-1}. 
\end{align}
On the other hand, we can bound the LHS from above as follows. 
Indeed, 
\begin{align}
 \int_{B_{R_1}(x)} |\psi|^2 \dv_g 
 \le&\int_{B_{R_1+|x|}(0)} |\psi|^2 \dv_g
    =\int_0^{R_1+|x|} H(0,r)\dd r.
\end{align}
By a change of variables~$\rho=\frac{R_1+|x|}{R_3} \rho$ where~$\rho\in [0,R_3]$, we get 
\begin{align}
 &\int_0^{R_1+|x|} H(0,r)\dd r
 =\int_0^{R_3} H(0,\frac{R_1+|x|}{R_3}\rho) \frac{R_1+|x|}{R_3}\dd\rho \\
 =&\parenthesis{\frac{R_1+|x|}{R_3}}^{n-2C_N}\int_0^{R_3} H(0,\rho)\exp\parenthesis{\int_{\rho}^{\frac{R_1+|x|}{R_3}\rho} \frac{2(N(0,t)+C_N)}{t}\dd t } \\ &\qquad\qquad\qquad \qquad \qquad \qquad \qquad\cdot \exp\parenthesis{\int_\rho^{\frac{R_1+|x|}{R_3}\rho} \mathscr{W}(0,t)\dd t }\dd\rho.
\end{align}
Note that 
\begin{align}
 \int_\rho^{\frac{R_1+|x|}{R_3}\rho} \mathscr{W}(0,t)\dd t
 \le \frac{\CCoord }{2}\parenthesis{\parenthesis{\frac{R_1+|x|}{R_3}}^2 - 1}\rho^2 \le \frac{\CCoord }{2}(R_1+|x|)^2 \le\frac{\CCoord }{2}
\end{align}
where we have used the conventional assumption~$r_3=1$, and that
\begin{align}
 \int_{\rho}^{\frac{R_1+|x|}{R_3}\rho} \frac{2(N(0,t)+C_N)}{t}\dd t
 \le&2(N(0,1)+ C_N)e^{\frac{C_{am}}{\beta+1} (1-\rho^{\beta+1})} \int_\rho^{\frac{R_1+|x|}{R_3}\rho}\frac{\dd t}{t} \\
 \le&2(N(0,1)+ C_N)e^{\frac{C_{am}}{\beta+1} }\ln\parenthesis{\frac{R_1+|x|}{R_3}}.
\end{align}
Hence, denoting~$A_2\equiv2(N(0,1)+ C_N)e^{\frac{C_{am}}{\beta+1} } $, we have 
\begin{align}\label{eq:upper bound for R1}
 \int_{B_{R_1}(x)}|\psi|^2\dv_g
 \le&e^{\frac{\CCoord }{2}(R_1+|x|)^2} \parenthesis{\frac{R_1+|x|}{R_3}}^{n-2C_N + A_2}\int_0^{R_3} H(0,\rho)\dd\rho \\
 =& e^{\frac{\CCoord }{2}(R_1+|x|)^2} \parenthesis{\frac{R_1+|x|}{R_3}}^{n-2C_N + A_2} \int_{B_{R_3}(0)} |\psi|^2\dv_g.
\end{align}
Combine~\eqref{eq:lower bound for R1} and~\eqref{eq:upper bound for R1} to obtain 
\begin{multline}\label{eq:comparison in step1}
 e^{-\frac{\CCoord }{2}(R_1^2 -R_2^2)}\frac{R_2 H(x,R_2)}{n-2C_N+A_1}\parenthesis{\parenthesis{\frac{R_1}{R_2}}^{n-2C_N+A_1}-1} \\
 \le  e^{\frac{\CCoord }{2}(R_1+|x|)^2} \parenthesis{\frac{R_1+|x|}{R_3}}^{n-2C_N + A_2} \int_{B_{R_3}(0)} |\psi|^2\dv_g, 
\end{multline}
with
\begin{align}
 A_1\equiv 2e^{\frac{C_{am}}{\beta+1}(R_2^{\beta+1} - R_1^{\beta+1})}(N(x, R_2)+C_N), & & 
 A_2\equiv2(N(0,1)+ C_N)e^{\frac{C_{am}}{\beta+1} }.
\end{align}

\noindent\textbf{Step 2.}
If~$B_{R_3}(0)\subset B_{R_2}(x)$ then, by the coarea formula 
\begin{align}
  &\int_{B_{R_3}(0)} |\psi|^2\dv_g 
 \le \int_{B_{R_2}(x)} |\psi|^2\dv_g 
    =\int_0^{R_2} H(x,r)\dd r \\
 =& \int_0^{R_2} H(x,R_2)\parenthesis{\frac{r}{R_2}}^{n-1-2C_N}
 \exp\parenthesis{\int_{R_2}^r \frac{2(N(x,t)+C_N)}{t}\dd t}
 \exp\parenthesis{\int_{R_2}^r \mathscr{W}(x,t)\dd t}\dd r \\
 =&\frac{R_2 H(x,R_2)}{R_2^{n-2C_N}}\int_0^{R_2}r^{n-1-2C_N}\exp\parenthesis{\int_r^{R_2} - \frac{2(N(x,t)+C_N)}{t}\dd t}
 \exp\parenthesis{\int_r^{R_2}- \mathscr{W}(x,t)\dd t}\dd r. 
\end{align}
Since 
\begin{align}
 \int_r^{R_2}- \mathscr{W}(x,t)\dd t
 \le \frac{\CCoord }{2}(R_2^2 - r^2) \le \frac{\CCoord }{2}R_2^2, 
\end{align}
\begin{align}
 \int_r^{R_2} - \frac{2(N(x,t)+C_N)}{t}\dd t
 \le&\int_r^{R_2}\frac{-2}{t}(N(x,0)+C_N) e^{\frac{C_{am}}{\beta+1}(-t^{\beta+1})} \\
 =&2(N(x,0)+C_N) e^{-\frac{C_{am}}{\beta+1}R_2^{\beta+1}} \ln\frac{r}{R_2}
 =A_3 \ln\frac{r}{R_2}, 
\end{align}
where~$A_3\equiv 2(N(x,0)+C_N) e^{-\frac{C_{am}}{\beta+1}R_2^{\beta+1}}$, 
we have the estimate 
\begin{align}
 \int_{B_{R_3}(0)}|\psi|^2 \dv_g 
 \le& e^{\frac{\CCoord }{2}R_2^2}\frac{R_2 H(x,R_2)}{n-2C_N + A_3}. 
\end{align}
Substituting this into~\eqref{eq:comparison in step1}, we obtain 
\begin{align}\label{eq:estimate for N-1}
 \parenthesis{\frac{R_1}{R_2}}^{n-2C_N+ A_1} 
 \le 1+ e^{\frac{\CCoord }{2}[(R_1+|x|)^2+R_1^2]}\parenthesis{\frac{R_1+|x|}{R_3}}^{n-2C_N+A_2} \frac{n-2C_N+A_1}{n-2C_N+A_3}. 
\end{align}
Moreover, by the almost monotonicity of frequency function, 
\begin{align}
 A_3= 2e^{-\frac{C_{am}}{\beta+1}R_2^{\beta+1}} (N(x,0)+C_N)
 \le 2 (N(x,R_2)+C_N) =  e^{\frac{C_{am}}{\beta+1}(R_1^{\beta+1} - R_2^{\beta+1})} A_1. 
\end{align}
In particular, since both~$R_1$ and~$R_2$ are small (since~$\rAM $ should be small, although we notationally take~$\rAM =1$ here), we have 
\begin{align}
 \frac{n-2C_N+ A_1}{n-2C_N + A_3} \in (\frac{1}{2}, 1+A_1). 
\end{align}
Thus from\eqref{eq:estimate for N-1} we get 
\begin{align}
 \parenthesis{\frac{R_1}{R_2}}^{n-2C_N+ A_1} 
 \le& 3 e^{\frac{\CCoord }{2}[(R_1+|x|)^2+R_1^2]}\parenthesis{\frac{R_1+|x|}{R_3}}^{n-2C_N+A_2}\frac{n-2C_N+A_1}{n-2C_N+A_3}  \\
 \le&  3A_1 e^{\CCoord }\parenthesis{\frac{R_1+|x|}{R_3}}^{n-2C_N+A_2}, 
\end{align}
which implies 
\begin{multline}
 (n-2C_N+ 2e^{\frac{C_{am}}{\beta+1}(R_2^{\beta+1}-R_1^{\beta+1})}(N(x,R_2)+C_N) ) \ln\frac{R_1}{R_2} \\
 \le \ln(N(x,R_2)+C_N) + \CCoord + \ln 6 
   + (n-2C_N + 2e^{\frac{C_{am}}{\beta+1}} (N(0,1)+C_N)) \ln\frac{R_1+|x|}{R_3}\,,
\end{multline}
and then
\begin{align}
 N(x,R_2)+C_N \le C(N(0,1)+C_N+1)
\end{align}
for some~$C=C(R_1, R_2, R_3, C_{am},\beta)$. 

Finally, to see that the above estimate is uniform for~$x\in B_{1/4}(0)$, we may take~$\eps=\frac{1}{8}$,
\begin{align}
 \rho\equiv \frac{1}{2}-2\eps -|x| =\frac{1}{4}-|x|, 
\end{align}
and 
\begin{align}
 R_1 =\frac{1}{2}+\rho +\eps, 
 & & 
 R_2=\frac{1}{2}+\rho, 
 & & 
 R_3=\eps=\frac{1}{8}. 
\end{align}
This conclude the proof.

\end{proof}

\begin{cor}
 Let~$(M^n,g)$ be a compact manifold and~$\psi$ solve~\eqref{eq:Dirac eq-V}, with~$V$ of type \eqref{eq:Vf} or \eqref{eq:Vq}. 
 Let~$N_\psi(x,r)$ denote the corresponding frequency function. 
 Then the set 
 \begin{align}
  \braces{N_\psi(x,0)\mid x\in M}
 \end{align}
 is unifomly bounded in~$\R_+$ (say by~$\cO_{\max}$). 
\end{cor}
This follows by a stardard covering argument, where the cover is finite since~$M$ is compact.

\section{Stratification of the nodal set}\label{sec:stratification}
Before proving Theorem \ref{thm:Stratification} we need to establish various technical results. 
Our strategy is motivated from that in \cite{Han1994singular}, where the author deals with the singular set of a scalar function solving a second order elliptic equation.
\medskip

Let~$\psi \in C^{1,\alpha}(B_1,\C^{N})$ be a solution to \eqref{eq:Dirac eq-f}, under the assumptions of \ref{thm:Stratification}. 
\medskip
Thanks to \ref{lem:Z1estimate} we only need to deal with $\cZ_{\geq 2}$, that is, with points where the spinor vanishes at least to order two.

Recall that 
\[
\cZ(\psi)=\bigcup_{k\geq 1}\cZ_k(\psi)\,.
\]
Then by \ref{lem:Nunifbound} one sees that for any $r\in(0,1)$ there exists an integer $\overline k =\overline k(r)$ such that
\[
\cZ_k(\psi)\cap B_r=\emptyset\,,\qquad \forall k>\overline k\,.
\]
Since locally in~$B_r$ we have~$\psi=\beta(\varphi)$, see Section 2.1, where~$\beta$ is a bundle isomorphism, we have 
\begin{align}
 \cZ_k(\psi)\cap B_r = \cZ_k(\varphi)\cap B_r. 
\end{align}
This fact will be used implicitly in this section and helps us to get control of~$\cZ(\psi)$ via~$\varphi$; loosely speaking one can identify~$\psi$ locally with~$\varphi$ without loss of much information. 

\

We are interested in studying the local behavior of $\psi$ near each point in the nodal set. To this aim, for each $p\in \cZ(\psi) \cap B_{1/2}$, for any $r\in(0,(1-\vert p\vert)/2)$ we define 
\begin{equation}\label{eq:rescalepsi}
\psi_{p,r}(x):=\frac{\psi(p+rx)}{\left(\fint_{\p B_r(p)} \vert \psi\vert^2\,\ds_r\right)^{1/2}}\,,\qquad x\in B_2\,.
\end{equation}
Using \ref{prop:decomp-k}, we see that
\[
\psi_{p,r}\to \Psi_p \,,\qquad \mbox{in $L^2(B_2(0))$, as $r\to0^+$,}
\]
where $\Psi_p$ coincides with the leading term $P$ in the expansion \eqref{eq:decomp-k}, normalized so that $\Vert \varphi_p\Vert_{L^2(\p B_1)}=1$. 
As such, the spinor $\Psi_p$ solves $\D_{g_0}\Psi_p=0$, with respect to the Euclidean metric. 
\begin{dfn}
The spinor $\Psi_p$ defined above is called a \emph{homogeneous blow-up} of $\psi$ at the point $p$.
\end{dfn}
By the above argument, the blow-up is unique. 
Notice that this fact might not be true in full generality, but it is not unexpected in this case since the solution is already known to be $C^{1,\alpha}$ and we have the decomposition in \ref{lem:decomp-spinor}. 
When it does not cause ambiguity, we will simply denote the blow-up by $\Psi$, for simplicity.
\smallskip

Observe that since $\Psi$ is smooth Euclidean spinor, the vanishing order at its nodal points is understood in classical sense. Thus
\[
\cZ_k(\Psi)=\{x\in \R^n\,:\, \partial^\alpha \Psi^j(x)=0\,,\forall \vert \alpha\vert\leq k-1; \exists 1\le j\le N, \;\exists \vert\beta\vert=k+1\,, \partial^\beta\Psi^j(x)\neq0\}\,.
\]

\begin{prop}
$\cZ_k(\Psi)$ is a linear subspace and there holds
\begin{equation}\label{eq:Zinv}
\Psi(x)=\Psi(x+y)\,,\qquad\forall x\in\R^n\,,\forall y\in\cZ_k(\Psi)\,.
\end{equation}
\end{prop}
\begin{proof}
This is a standard fact for the homogeneous harmonic polynomials.
Indeed, since the components of $\Psi$ are homogeneous polynomials, it is immediate that $0\in\cZ_k(\Psi)$.
Take $y\in\cZ_k(\Psi)$ and $j=1,\ldots, N$, so that 
\[
\partial^\alpha\Psi^j(y)=0\,,\qquad \forall \vert \alpha\vert\leq k-1\,.
\]
Assuming that 
\[
\Psi^j(x)=\sum_{\vert\alpha\vert=k}a^j_\alpha x^\alpha\,,
\]
then
\[
\Psi^j(x)=\sum_{\vert \alpha\vert=k}a^j_\alpha(x-y)^\alpha\,,
\]
and hence
\[
\Psi(x+y)=\Psi(x)\,,\qquad\forall x\in\R^n,.
\]
Again, since each $\Psi^j$ is a homogeneous polynomial, then 
\[
\Psi(x+\lambda y)=\Psi(x)\,,\qquad\forall x\in\R^n,\lambda\in\R\,,
\]
and 
\[
\partial^\alpha \Psi(\lambda y)=0\,,\qquad \forall \vert \alpha\vert\leq k-1\,.
\]
Thus $\lambda\in \cZ_k(\Psi)$ for any $\lambda\in\R$. 
Then it is not hard to see that $\cZ_k(\Psi)$ is a \emph{linear subspace} and \eqref{eq:Zinv} holds.
\end{proof}

Observe that $\Psi$ is essentially a function of $(n-\dim\cZ_k(\psi))$-variables up to a change of variables, thanks to \eqref{eq:Zinv}.
We claim that
\begin{equation}\label{eq:Zdim}
\dim \cZ_k(\Psi)\leq n-2\,.
\end{equation}
Indeed, by contradiction, assume $\dim\cZ_k(\Psi)=n-1$. 
Then the nonzero components of $\Psi$ would be degree~$k$ monomials of one variable, say~$(x^1)^k$. 
Combining the fact that $\D\Psi=0$ and $\D^2=-\Delta$, on the Euclidean space, then one finds $-\Delta\Psi=0$, and thus $k=0$ or $k=1$. But this is impossibile, since we are assuming $k\geq2$.
\medskip

Now fix $y\in\cZ_k(\psi)$ and consider the corresponding homogeneous blow-up $\Psi_y$. Set
\begin{equation}\label{eq:ldim}
\cZ_k^l(\psi)=\{y\in\cZ_k(\psi)\,:\,\dim\cZ_k(\Psi_y)=l\}\,,
\end{equation}
for $l=0,1,\ldots,n-2$.

\begin{lemma}\label{lem:Pconv}
Let $(y_m)\subseteq\cZ_k(\psi)$ be a sequence of points such that $y_m\to y\in\cZ_k(\psi)$. 
Consider the leading terms $P^m_k, P_k$ of $\varphi=\beta^{-1}(\psi)$ in the decomposition \eqref{eq:decomp-k} at~$y_m$ and $y$, respectively. Then
\[
P^m_k\to P_k\,,\qquad \mbox{as $m\to\infty$}\,,
\]
in $C^k$-norm, on any compact subset of $\R^n$.
\end{lemma}

\begin{proof}
 Without loss of generality we assume~$y=0$ is the center of the local coordinate chart~$B_1\subset \R^n$. 
 Then 
 \begin{align}
  \varphi(x)= P_k(x)+ Q_k(x)
 \end{align}
 for some homogeneous harmonic~$P_k$, as in \ref{prop:decomp-k}. 
 Denote 
 \begin{align}
  \varphi_m(x)\coloneqq \varphi(y_m+x).
 \end{align}
 Then the corresponding decomposition for~$\psi_m$ at~$0$ is precisely given by~$P^m_k$. 
 
 Since~$\psi \in C^{1,\alpha}$, the spinors~$\varphi_m$ converges in~$C^{1,\alpha}(B_{1-\eps})$ for~$\eps>0$ small. 
 By the proof of \ref{lem:decomp-spinor} we see that~$P_k^m\to P_k$ uniformly in~$B_{1-{2\eps}}$.
 But the degree~$k$ homogeneous harmonic polynomials in~$\R^n$ constitute a finite dimesnional vector space, on which all the norms are equivalent, hence they also converge in~$C^k$. 
\end{proof}

\begin{rmk}
 The vanishiing order is well-known to be upper semi-continuous. 
 For a sequence~$y_m\in \cZ_k(\psi)$ which converges to~$y\in \cZ_l(\psi)$, then necessarily~$l\ge k$; and if~$l>k$, then actually~$P_k^m\to 0$ in~$C^k$. 
\end{rmk}

\begin{lemma}\label{lem:union}
The set $\cZ^l_k(\psi)$ is contained in the countable union of $l$-dimensional $C^1$ graphs, for any $l=0,1,\ldots, n-2$.
\end{lemma}
\begin{proof}
Given $y\in\cZ_k^l(\psi)$, let $\ell_y$ be the $l$-dimensional linear space $\cZ_k(\Psi_y)\subseteq \R^n$.

\noindent{\bf Step 1.} We start by proving that, for any sequence of points $(y_m)\subseteq \cZ_k^l(\psi)$ which converges to~$y$, there holds
\begin{equation}\label{eq:angle}
\operatorname{angle}\langle \overline{y y_m}, \ell_y\rangle\to 0\,,
\end{equation}
where by $\overline{y y_m}$ we denote the segment connecting the points $y$ and $y_m$.

\

We may assume $y=0$ and up to subsequence $\xi_m=y_m/\vert y_m\vert\to\mathbb \xi\in \mathbb S^{n-1}$.
Then $\xi_m$ is a zero of the spinor $\psi_{0,\vert y_m\vert}$, defined as in \eqref{eq:rescalepsi}, with vanishing order~$k$.
\medskip

We now claim that 
\begin{equation}\label{eq:rescaleconv}
\psi_{0,\vert y_m\vert}\to \Psi_y \,,\qquad \mbox{in $C^{1,\alpha}$, as $m\to \infty$,}
\end{equation}
and that
\begin{equation}\label{eq:xinodal}
\xi\in\cZ_k(\Psi_y)=\ell_y\,.
\end{equation}
Note that~\eqref{eq:xinodal} immediately implies~\eqref{eq:angle}.

Indeed, since~$\psi$ is a solution of~\eqref{eq:Dirac eq-V}, the local spinor~$\varphi=\beta^{-1}(\psi)$ solves the equation~\eqref{eq:EuclideanDirac-V}, while the local Euclidean spinor~$\varphi_{0,|y_m|}\coloneqq\beta^{-1}(\psi_{0,|y_m|})$ solves the following
\begin{align}
 \D_{g_0} \varphi_{0,|y_m|} (x)
 =&\frac{ |y_m| }{\parenthesis{ \fint_{\p B_{|y_m|}} |\psi|^2\ds_{|y_m|} }^{\frac{1}{2}}} \widetilde{V}(\varphi)(|y_m|x)  \\
  &-\sum_{i,j}(b_i^j(|y_m|x)-\delta_i^j) \gamma_{g_0}(\p_i)\snabla^0_{p_i}\varphi_{0,|y_m|} (x) \\
   &-\frac{|y_m|}{4}\sum_{i,j,k}\Gamma^k_{ij}(|y_m|x) \gamma_{g_0}(\p_i)\gamma_{g_0}(\p_j)\gamma_{g_0}(\p_k)\varphi_{0,|y_m|}(x)
\end{align}
Since~$y=0$ and~$|y_m|\to 0$, the right hand sides converges locally uniformly to zero. 
Thus~$\varphi_{0,|y_m|}$ converges in~$C^{1,\alpha}(B_2)$ to a harmonic spinor, which is~$\Psi_y$; and so is~$\psi_{0,|y_m|}=\beta^{-1}(\varphi_{0,|y_m|})$. 
In particular, 
\begin{align}
 \limsup_{m\to+\infty} \cO_{\psi_{0,|y_m|}}(\xi_m)+C_N 
 \le& \limsup_{m\to+\infty} \exp\parenthesis{ \frac{\CAM}{\beta+1} r^{\beta+1}} \parenthesis{N_{\psi_{0,|y_m|}} (\xi_m, r) + C_N } \\
  = &\exp\parenthesis{ \frac{\CAM}{\beta+1} r^{\beta+1}} \parenthesis{N_{\Psi_y} (\xi, r) + C_N }
\end{align}
and then sending~$r\to 0+$ we see that 
\begin{align}
 k=\limsup_{m\to+\infty} \cO_{\varphi_{0,|y_m|}}(\xi_m)
 \le \cO_{\Psi_y}(\xi),
\end{align}
which is essentially the upper semi-continuity of the vanishing order, defined via frequency functions.  
But~$\Psi_y$ is a homogeneous harmonic polynomial of degree~$k$, so~$\cO_{\Psi_y}(\xi)=k$, thus~$\xi\in \cZ(\Psi_y)=\ell_y$.

\noindent{\bf Step 2.} 
By \eqref{eq:angle} we see that, for any $y\in\cZ_k^l(\psi)$ and~$\eps>0$ small, there exists~$r=r(y,\eps)>0$ such that
\begin{equation}\label{eq:localZ}
\cZ^l_k(\psi)\cap B_r(y)\subseteq B_r(y)\cap \mathscr{C}_\eps(\ell_y)\,,
\end{equation}
where~$\mathscr{C}(\ell_y)$ is the~$\eps$-cone around~$\ell_y$:
\[
\mathscr{C}_\eps(\ell_y):=\{z\in\R^n\,:\, \dist(z,\ell_y)\leq\eps\vert z\vert\}\,.
\]
Let $P^m_k$ and $P_k$ be the leading polynomial of $\psi$ at $y_m$ and at $y$, respectively. By Lemma \ref{lem:Pconv} one sees that 
\[
\ell_{y_m}\to\ell_y\,,\qquad \mbox{as $m\to\infty$}\,,
\]
as subspaces. 
Then \eqref{eq:localZ} holds uniformly near $y$, in the sense that for any $y\in\cZ_k^l(\psi)$ and for any $\eps>0$ small, there exists a radius $r=r(y,\eps)>0$ such that
\[
\cZ^l_k(\psi)\cap B_{r}(y)\subseteq C_{\eps}(\ell_z)\cap  B_r(y) \,,\qquad \mbox{for any }\; z\in\cZ_k^l(\psi)\cap B_r(z)\,.
\]
Thus, taking $\eps>0$ small enough this implies that $\cZ^l_k(\psi)\cap B_r(y)$ is contained in a $l$-dimensional Lipschitz graph, which in addition is $C^1$ by \eqref{eq:angle}.
\end{proof}

The regularity of the top-dimensional stratum $\cZ^{n-2}_k(\psi)$ can be improved.
\begin{lemma}\label{lem:n-2}
$\cZ^{n-2}_k(\psi)$ is contained in the countable union of $(n-2)$-dimensional $C^{1,\alpha'}$ manifolds, for some $0<\alpha'<1$.
\end{lemma}
\begin{proof}
The proof is similar to \cite[Lemma 2.4]{Han1994singular}, and we give it here for completeness.  
Let $y\in\cZ^{n-2}_k(\psi)$. 
Take a local chart so that~$y=0$, and let $P_k$ be the spinor in \eqref{eq:decomp-k} near $y=0$, whose components are either zero or homogeneous harmonic polynomials of degree $k$, depending on two variables. Decomposing
\[
x=(x_1,x_2)\in \R^2\times \cZ_k(\varphi_0)=\R^n\,,
\]
consider $x\in\cZ_k^{n-2}(\psi)$ close to the origin~$y=0$. Since~$k\ge 2$, we have~$\nabla \psi(x)=0$, and hence
\[
\nabla P_k(x)=-\nabla (\varphi-P_k)(x)\,,
\]
and by \eqref{eq:decomp-k} we see that
\[
\vert x_1\vert^{k-1}\leq C \vert x\vert^{k-1+\delta} \,,
\]
that is 
\[
\vert x_1\vert \leq C \vert x\vert^{1+\frac{\delta}{k-1}}\,,
\]
where $C>0$ is a constant. By Lemma \ref{lem:union}, the local $(n-2)$-dimensional $C^1$ manifold containing the set $\cZ^{n-2}_k(\psi)$ near $p=0$ is actually of class $C^{1,\alpha'}$ with~$\alpha'\le \frac{\delta}{k-1}$.
\end{proof}

Collecting the above results we get the following
\begin{proof}[Proof of Theorem \ref{thm:Stratification}]
Recall that there holds
\[
\cZ(\psi)=\bigcup_{k\geq 1}\cZ_k(\psi)=\bigcup_{k\geq1}\bigcup_{l=0}^{n-2}\cZ^l_k(\psi)\,,
\]
where $\cZ_k(\psi)$ is the set of points where the spinor vanishes to order $k$. 
The sets $\cZ_k^l(\psi)$ are defined in \eqref{eq:ldim}. 
The claim thus follows combining Lemma \ref{lem:Z1estimate}, dealing with points of order $k=1$, while the sets $\cZ^l_k(\psi)$, $ 2\leq k\leq \cO_{\max}, l=0,\ldots, n-2$ are considered in Lemma \ref{lem:union} and finally, noting that~$2\leq k\leq \cO_{\max}<+\infty$ so that~$\alpha'$ can be uniformly chosen, Lemma \ref{lem:n-2} concludes the proof. .

\end{proof}

\end{document}